\documentclass{article}
\usepackage{amsmath}
\usepackage{amssymb}
\usepackage[utf8]{inputenc}
\usepackage[T1]{fontenc}
\usepackage{amsthm}
\usepackage{xcolor}
\usepackage{soul}
\usepackage{hyperref}
\usepackage{geometry}
\usepackage{lmodern}
\usepackage{float}

\newtheorem{theorem}{Theorem}[section] 
\newtheorem{lemma}[theorem]{Lemma}
\newtheorem{proposition}[theorem]{Proposition}
\newtheorem{corollary}[theorem]{Corollary}
\newtheorem{assumption}{Assumption}

\theoremstyle{remark}
\newtheorem{remark}{Remark}[section]

\theoremstyle{definition}
\newtheorem{example}{Example}[section]

\newcommand{\ori}{\textup{o}}
\newcommand{\deriv}{\mathrm{d}}
\newcommand{\esp}{\mathbf E}
\newcommand{\espn}{\mathbf E^n}

\usepackage[textwidth=2.5cm]{todonotes}
\presetkeys{todonotes}{size=\scriptsize}{}

\newcommand{\Emmett}[5]{
\draw[#4] (0,0)
\foreach \x in {1,...,#1}
{   -- ++(#2,rand*#3)
}
node[right] {#5};
}

\begin{document}
\title{Averaging principle for jump processes depending on fast ergodic dynamics}

\author{Vincent Kagan$^{1}$, Edouard Strickler$^{1}$, Denis Villemonais$^{2,3}$}

\footnotetext[1]{Universit\'e de Lorraine, CNRS, Inria, IECL, F-54000 Nancy, France \\
  E-mail: vincent.kagan@univ-lorraine.fr, edouard.strickler@univ-lorraine.fr}
  
  \footnotetext[2]{Université de Strasbourg,  IRMA, Strasbourg, France\\ E-mail : denis.villemonais@unistra.fr}
  
  \footnotetext[3]{Institut Universitaire de France}
\maketitle

\begin{abstract}
    We consider a  slow-fast stochastic process  where the slow component is a jump process on a measurable index set  whose transition rates depend on the position of the fast component. Between the jumps, the fast component evolves according to an ergodic dynamic in a state space determined by the index process.
 We prove that, when the ergodic dynamics are accelerated, the slow index process converges to an autonomous pure jump process on the index set. 
 
 We apply our results to prove the convergence of a typed branching process toward a continuous-time Galton–Watson process, and of an epidemic model with fast viral loads dynamics to a standard contact process.
\end{abstract}

\section{Introduction}

The averaging principle is a cornerstone of the analysis of slow–fast stochastic systems, providing a  framework for reducing high-dimensional models by exploiting time-scale separation.
In his foundational works, notably \cite{Hasminskii1966a} and \cite{Hasminskii1966b},  Has’minskii demonstrated that, under the regime where fast components of a stochastic system are increasingly accelerated, the system’s behaviour may be effectively described by averaged or reduced dynamics—often deterministic or exhibiting pure jump characteristics. His approach formalises the transition from detailed multi-component dynamics to effective equations governing slow variables.

Building on this perspective, Kurtz \cite{KurtzAveraging}
showed
 that averaging or reduced dynamics of such multiscale stochastic systems can be formulated via martingale problems. Related perspectives in Freidlin–Wentzell theory \cite{book} show that metastable fast dynamics naturally give rise to limiting pure jump processes. 
Classical results on averaging principles also relate to diffusion approximations, see e.g.~\cite{PardouxYu}.
In \cite{FreidlinKoralov,Goddard1}, the authors investigate the subtleties that arise when the fast process possesses multiple invariant measures.
We also refer the reader to  \cite{HairerLi} for averaging dynamics driven by fractional Brownian motion, to
\cite{Yin} for singular perturbations in switching diffusions, and to \cite{Genadot} for piecewise deterministic Markov processes with attractive boundaries.

In this paper, we consider a two-timescales stochastic process $(\mathbf X^n_t,\mathbf I^n_t)_{t\geq 0}$, where $\mathbf I^n$ is a (slow) jump process on an index set $I$ whose transition rates depend on the position of the (fast) process $\mathbf X^n$. Between the jumps of $\mathbf I^n$, the process $\mathbf X^n$ evolves according to an ergodic dynamic in a state space which depends on $\mathbf I^n$.
 We prove that, when the ergodic dynamics are accelerated, the process $\mathbf{I}^n$  converges to an autonomous pure jump process $\mathbf{I}$ on the index set $I$. 

The limiting dynamics are often straightforward to describe, and they connect rigorously to classical models long used in applications. To illustrate this principle, we provide two applications: (i) the convergence of a typed branching process toward a continuous-time Galton–Watson process, and (ii) an epidemic model on a graph with viral loads that, under accelerated viral dynamics, converges to a standard contact process.

Our results situate within a broad literature on multiscale stochastic modeling, encompassing chemical reaction networks, interacting particle systems, and biological dynamics. For instance, recent works, such as \cite{CentralLimit,AsymptoticAnalysis,AnalysisOfStochastic}, employed averaging and Poisson equation techniques to analyze chemical networks, while related approaches have been applied to biological models \cite{ballif:hal-03405177} and structural classes of reaction systems \cite{Product-FormStationary}. Other contributions have explored mean-field limits in rapidly varying environments \cite{AParticleSystem}. By extending these perspectives to (non-necessarilly Markov) stochastic processes with ergodic fast components, we provide a unified methodology for deriving pure jump limits, and thereby a rigorous foundation for using classical processes to approximate more complex multiscale dynamics.

Our proofs rely on fundamental properties of ergodic stochastic dynamics, escaping the technicalities of martingale problems and infinitesimal generators as developed in~\cite{KurtzAveraging}. In particular, we identify the limit as an explicit pure jump process, whereas, in the latter article, the author proves a tightness result and  characterises the limit points as solutions of a martingale problem.  In addition, our methods allow us to extend the classical Markov setting in Polish spaces to non-Markov stochastic processes evolving in general state spaces. 
Note that however our results only apply to two-time scales processes while the results in~\cite{KurtzAveraging} apply to a broader class of averaging problems (e.g. space rescaling).

The paper is organised as follows. In Section~\ref{sec:base}, we give the precise definition of the stochastic process evolving in the different communication classes, while its rigorous construction is postponed to Appendix~\ref{sec:appendixPO}. The main convergence results for the accelerated process are given in Section~\ref{sec:main-results} and the proofs in Section~\ref{sec:proofs}. Section~\ref{sec:appli} presents some applications of our results.

\section{Base process and first properties}
\label{sec:base}
Let $(E_i,\mathcal{E}_i)_{i\in I}$ a collection of measurable spaces indexed by a measurable set $(I,\mathcal I)$.  We define the measurable space $(E,\mathcal{E})$ with $E:=\sqcup_{i\in I}E_i$ (here $\sqcup$ denotes the disjoint union) endowed with the $\sigma$-algebra $\mathcal{E}$  defined by
\[
\mathcal{E}:=\{S\subset E:S\cap E_i\in\mathcal{E}_i\;\textup{for each}\;i\in I\}
\]
In particular, with this choice of $\sigma$-algebra, the application 
\[
\phi:x\in E\to i\in I\text{ such that }x\in E_i
\]
is measurable (see Lemma~\ref{lem:mesE} in the Appendix).

 In the remainder of this section, we describe an $E\times I$-valued stochastic process $(\mathbf X^1, \mathbf I^1)$, where $\mathbf I^1$ is a jump process on $I$ whose transitions depend on the position of $\mathbf X^1$ and, when $\mathbf I^1$ is in $i\in I$, $\mathbf X^1$ evolves on $E_i$ according to some progressively measurable dynamic up to the next jump time of $\mathbf I^1$. Then $(\mathbf X^1, \mathbf I^1)$ transitions from $E_i \times \{i\}$ to a new position in $E \times I$.
 
 In order to describe our base stochastic process $(\mathbf{X}^1, \mathbf I^1)$, we need three ingredients: the dynamic of $\mathbf X^1$ on each $E_i$ between the jumps of $\mathbf I^1$, the law of the jump times, and the law of the  new position in $E \times I$ after the jump. Note that it is actually sufficient to describe the dynamic of $\mathbf X^1$ on $E$, since $\mathbf I^1=\phi(\mathbf X^1)$.

\paragraph{Dynamic of $\mathbf X^1$ between the jumps of $\mathbf I^1$} Let $(V,\mathcal F^\ori)$ be a measurable space endowed with a filtration $(\mathcal F^\ori_t)_{t\geq 0}$ and a family of probability measures $(\mathbf{P}_x^\ori)_{x\in E}$. We assume that we are given a progressively measurable  process 
\[
X^\ori=(V, \mathcal F^\ori,(\mathcal F^\ori_t)_{t\geq 0},(\mathbf{P}_x^\ori)_{x\in E},(X_t^\ori)_{t\geq0})
\]
with values in the measurable space $(E,\mathcal E)$ and such that, for all $i\in I$ and $x_i\in E_i$,
\begin{equation}
\label{eq:stable-E-i}
\mathbf{P}_{x_i}^\textup{o}(\forall t\geq0,X_t^\ori\in E_i)=1.
\end{equation}

\begin{remark}
Note that condition~\eqref{eq:stable-E-i} can also be interpreted as the fact that, on each $E_i$, we are given a stochastic process $X^{\ori,i}$. This is reminiscent of Markov processes with switching (see e.g. \cite{BH12, BMZIHP, CDGMMY17, YZ09, CH15, CKW21}  when $I$ is finite or countable). Our construction relies on a construction by piecing-out inspired by \cite{MBP2} in the context of branching processes.
\end{remark}

\paragraph{Law of the jump time of $\mathbf{I}^1$}
We consider an absolutely continuous cumulative distribution function $G:\mathbb{R}_+\longrightarrow [0,1]$, and a measurable rate function $b: E\rightarrow [0,\infty)$ such that $\int_{0+} b(X^\ori_t)\,\mathrm dt<\infty$ $\mathbf P^\ori_x$-almost surely, for all $x\in E$. For all $x \in E$, we let $\nu_x$ denote the distribution of the random variable 
\[
\zeta = \inf\{ t \geq 0 \: : G\Big(\int_0^t b (X_s^\ori) \, \deriv s\Big)\geq  U \}
\]
under $\mathbf P_x^\ori$, and where $U$ is an independent random variable with uniform law on $[0,1]$. The idea is that $\nu_x$ is the law of the time elapsed between the $k$-th and the $(k+1)$-th jump times whenever the process is at $x$ after the $k$-th jump time. Note also that in the particular case where $G(x) = 1 - e^{-x}$, we recover the classical setting of a process with exponential jump rate~$b$.

\paragraph{Transition kernel}  We also assume that we are given a transition kernel $\pi$ from $E$ to $E$: if the process is at some position $x \in E$ just before the transition time, then its new position is chosen according to $\pi(x, \cdot)$.

\medskip

In the following proposition, $\Delta\notin E$ plays the role of a cemetery point.
\begin{proposition}
There exists a random process $\mathbf{X}^1 = (\Omega,\mathcal{F}^1,(\mathbf{X}_t^1)_{t\geq0}, (\mathbf{P}_x^1)_{x\in E})$ with state space $E\cup \{\Delta\}$ with lifetime $T_\Delta^1:=\inf\{t\geq 0,\mathbf X_t^1=\Delta\}$, and a random sequence of times $(\tau_k^1)_{k \geq 0}$ such that
\begin{enumerate}
    \item $\tau_0^1 = 0$;
    \item for all $k \geq 0$, conditionally on $\mathbf{X}_{\tau_k^1}^1$
    the distribution function of $\tau_{k+1}^1 - \tau_k^1$ is given by $\nu_{\mathbf{X}_{\tau_k^1}^1}$;
    \item for all $k \geq 0$, conditionally on $\mathbf{X}_{\tau_k^1}^1$,  $(\mathbf{X}_{t-\tau_k^1}^1)_{t\in [\tau_k^1, \tau_{k+1}^1)}$ is distributed as $(X^\ori_t)_{t\in[0,\zeta)}$ under $\mathbf P^\ori_{\mathbf{X}_{\tau_k^1}^1}$;
    \item at time $t = \tau_{k+1}^1$, $\mathbf{X}_t^1$ transitions to a random position selected according to the kernel $\pi(\mathbf{X}_{\tau_{k+1}^1-}^1,\,\cdot\,)$;
    \item $T_\Delta^1=\lim_{n\to+\infty} \tau_k^1$.
\end{enumerate}
In addition, this process satisfies the strong Markov property at its transition times $(\tau_k^1)_{k \geq 0}$.
\end{proposition}

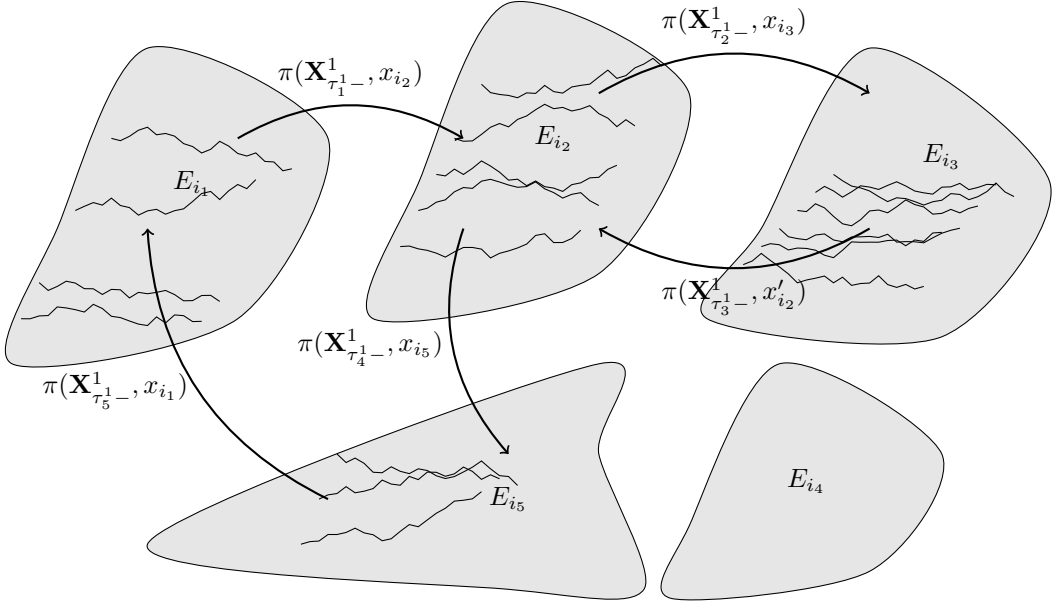
\begin{figure}[h]
    \centering
\pgfmathsetseed{2025}
\begin{tikzpicture}[scale=1.2]
    \newcommand{\totalwidth}{16cm}
    \newcommand{\halfwidth}{\totalwidth / 2}
    
    \begin{scope}[shift={(-\halfwidth + \textwidth / 2, 0)}]

        \filldraw[fill=gray!20, draw=black] plot [smooth cycle] coordinates {(-1,0) (0,2) (2,1) (1,-1) (-1.5,-1.5)};

        \node at (0.5,0.5) {$E_{i_1}$};
        \foreach \i in {-1,1,5,9} {
            \begin{scope}[shift={(\i*0.1-1.3,\i*0.2-0.8)}]  
                \Emmett{20}{0.1}{0.1}{black}{}
            \end{scope}
        }
        \begin{scope}[xshift=4cm]
        \filldraw[fill=gray!20, draw=black] plot [smooth cycle] coordinates {(-1,0.5) (0,2.5) (2,1.5) (1,-0.5) (-1.5,-1)};
        \node at (0.5,1) {$E_{i_2}$};
        \foreach \i in {1,3,5,7,10} {
            \begin{scope}[shift={(\i*0.1-1.3,\i*0.2-0.4)}]  
                \Emmett{20}{0.1}{0.1}{black}{}
            \end{scope}
        }
        \end{scope}

        \begin{scope}[xshift=8cm]
        \filldraw[fill=gray!20, draw=black] plot [smooth cycle] coordinates {(-1.2,0) (0,2) (2,0.5) (0.8,-1.2) (-1.8,-1)};
        \node at (0.8,0.8) {$E_{i_3}$};
        \foreach \i in {1,...,6} {
            \begin{scope}[shift={(\i*0.2-1.6,\i*0.2-0.6)}]  
                \Emmett{20}{0.1}{0.1}{black}{}
            \end{scope}
        }
        \end{scope}

        \begin{scope}[xshift=7cm, yshift=-3.7cm]
        \filldraw[fill=gray!20, draw=black] plot [smooth cycle] coordinates {(-1,0.4) (0,2.2) (1.8,1.2) (1,-0.1) (-1.2,-0.4)};
        \node at (0.3,0.9) {$E_{i_4}$};
        \end{scope}

        \begin{scope}[xshift=5cm, yshift=-4cm]
        \filldraw[fill=gray!20, draw=black] plot [smooth cycle] coordinates {(-5,0.5) (0,2.5) (0,1.5) (0.5,0) (-1,0)};
        \node at (-1,1) {$E_{i_5}$};
        \foreach \i in {1,...,3} {
            \begin{scope}[shift={(\i*0.2-3.5,\i*0.5-0.)}]  
                \Emmett{20}{0.1}{0.1}{black}{}
            \end{scope}
        }
        \end{scope}
        
        \draw[->, thick, bend left] (1,1) to node[midway, above] {$\pi(\mathbf{X}_{\tau_1^1-}^1,x_{i_2})$} (3.5,1);
        \draw[->, thick, bend left] (5,1.5) to node[midway, above] {$\pi(\mathbf{X}_{\tau_2^1-}^1,x_{i_3})$} (8,1.5);
        \draw[<-, thick, bend right] (5,0) to node[midway, below] {$\pi(\mathbf{X}_{\tau_3^1-}^1,x_{i_2}')$} (8,0);
        \draw[->, thick, bend right] (3.5,0) to node[midway, left] {$\pi(\mathbf{X}_{\tau_4^1-}^1,x_{i_5})$} (4,-2.5);
        \draw[<-, thick, bend right] (0,0) to node[midway, left] {$\pi(\mathbf{X}_{\tau_5^1-}^1,x_{i_1})$} (2,-3);
        
    \end{scope}
\end{tikzpicture}
    \caption{Schematic representation of the dynamic of $\mathbf X^1$, built by piecing out and evolving according to specified stochastic dynamics on each $E_i$, $i\in I$.}
    \label{fig:trajectory}
\end{figure}

The process $\mathbf X^1$ is built formally in Section~\ref{sec:appendixPO}, following the construction by piecing out of \cite{MBP2}, see Figure~\ref{fig:trajectory} for a schematic representation of the trajectories of $\mathbf X^1$. As for the index process, we set $\mathbf I^1 = \phi( \mathbf X^1)$, where it is understood that $\phi(\Delta) = \partial$ for some $\partial \notin I$. Technical elementary properties of this process are also provided in Section~\ref{sec:appendixPO}. We conclude this section with three propositions that will be key elements of the proof of the main results.

\begin{proposition}
\label{prop:moyenne1}
Let  $f:E\cup\{\Delta\}\longrightarrow\mathbb{R}$ be a bounded measurable function. Then, for all $x\in E$ and $T>0$,
\[
\mathbf E_x^1\big[f(\mathbf{X}_{\tau_1^1}^1)\,\mathbf{1}_{\{\tau_1^1\leq T\}}\big]=\mathbf{E}_x^\ori\Big[\int_0^T\deriv t\,G'\Big(\int_0^tb(X_{s}^\ori)\,\deriv s\Big)\,b(X_{t}^\ori)\int_E\pi(X_{t}^\ori,\deriv y)\,f(y)\Big].
\]
\end{proposition}

\begin{proof}
The result is an immediate application of Proposition~\ref{prop:piecingOut}. Indeed, according to this result,
\begin{align*}
\mathbf E_x\big[f(\mathbf{X}_{\tau_1^1}^1)\,\mathbf{1}_{\{\tau_1\leq T\}}\big]
&=\mathbf E^\ori_x\big[\pi(X^\ori_{\zeta},f)\,\mathbf 1_{\zeta\leq T}\big], 
\end{align*}
where, under the assumption that $G$ is absolutely continuous and conditionally to $(X^\ori_t)_{t\geq 0}$, the law of $\zeta$  has density $G'\big(\int_0^tb(X_{s}^{\ori}(v))\,\deriv s\big)\,b(X_{t}^{\ori}(v))$. This concludes the proof of Proposition~\ref{prop:moyenne1}.
\end{proof}

\begin{proposition}
\label{prop:moyenneN}
Let $N\geq1$, $T_1,\ldots,T_N>0$ 
and  $f:E^N\to\mathbb R$ be a bounded $\mathcal{E}^{\otimes N}$-measurable function. Then, for all $x\in E$ we have
\begin{align*}   
\mathbf E_x^1\big[f(\mathbf{X}_{\tau_1^1}^1,\ldots&,\mathbf{X}_{\tau_N^1}^1)\,\mathbf{1}_{\{\tau_1^1\leq T_1\}}\ldots \,\mathbf{1}_{\{\tau_N^1-\tau_{N-1}^1\leq T_N\}}\big]\\
&=\mathbf{E}_x^\ori\Big[\int_0^{T_1}\deriv t_1\,G'\Big(\int_0^{t_1}\,b(X_{s}^\ori)\,\deriv s\Big)\,b(X_{t_1}^\ori)\int_E\pi(X_{t_1}^\ori,\deriv y_1)\mathbf{E}_{y_1}^\ori\Big[\ldots\hspace{2cm}\\
&\times\mathbf{E}_{y_{N-1}}^\ori\Big[\int_0^{T_N}\deriv t_N\,G'\Big(\int_0^{t_N}b(X_{s}^\ori)\,\deriv s\Big)\,b(X_{t_N}^\ori)\,\pi(X_{t_N}^\ori,\deriv y_N)\,f(y_1,\ldots,y_N)\Big]\ldots\Big]\Big].
\end{align*}
\end{proposition}

\begin{proof}
We prove the result by induction on $N\geq1$. Let $N\geq1$, $T_1,\ldots,T_{N+1}>0$ and $f_N:E^{N}\to\mathbb{R}$ a bounded $\mathcal E^{\otimes N}$-mesurable function. For all $N\geq1$, we set
\[
H_N(\mathbf{X}^1)=\mathbf{1}_{\{\tau_1^1-\tau_0^1\leq T_N\}}\ldots \,\mathbf{1}_{\{\tau_N^1-\tau_{N-1}^1\leq T_1\}}\,f_N(\mathbf{X}_{\tau_1}^1,\ldots,\mathbf{X}_{\tau_N^1}^1)
\]
with $\tau_0^1:=0$. With this notations, the induction hypothesis is given by the following equality
\begin{align*}
\esp_x^1\big[H_N(\mathbf{X}^1)\big]&=\mathbf{E}_x^\ori\Big[\int_0^{T_N}\deriv t_N\,G'\Big(\int_0^{t_N}\,b(X_{s}^\ori)\,\deriv s\Big)\,b(X_{t_N}^\ori)\int_E\pi(X_{t_N}^\ori,\deriv y_N)\\
&\times\mathbf{E}_{y_N}^\ori\Big[\ldots\mathbf{E}_{y_{2}}^\ori\Big[\int_0^{T_1}\deriv t_1\,G'\Big(\int_0^{t_1}b(X_{s}^\ori)\,\deriv s\Big)\,b(X_{t_1}^\ori)\\
&\times\int_E\pi(X_{t_1}^\ori,\mathrm dy_1)\,f_N(y_N,\ldots,y_1)\Big]\ldots\Big]\Big].
\end{align*}
For $N=1$, it is the previous proposition. In order to prove the property for $N+1$, we recall first that, for all  $k\geq0$, $\tau_{k+1}^1=\tau_k^1+\tau_1^1(\Tilde{\theta}_{\tau_k^1})$ (where $\Tilde{\theta}_{\tau_k^1}$ is the shift operator, see Section~\ref{sec:appendixPO} for details)
, and hence
\begin{align*}
H_{N+1}(\mathbf{X}^1)&=\mathbf{1}_{\{\tau_1^1-\tau_0^1\leq T_{N+1}\}}\ldots \,\mathbf{1}_{\{\tau_{N+1}^1-\tau_{N}^1\leq T_1\}}\,f_{N+1}(\mathbf{X}_{\tau_1^1}^1,\ldots,\mathbf{X}_{\tau_{N+1}^1}^1)\\
&=\mathbf{1}_{\{\tau_1^1-\tau_0^1\leq T_{N+1}\}}\,\tilde H_{N}(\Tilde{\theta}_{\tau_1^1}\mathbf{X}^1),
\end{align*}
where $f_{N+1}:E^{N+1}\to\mathbb R$ is a bounded $\mathcal E^{\otimes(N+1)}$-measurable function and 
\[
\tilde H_{N}(\mathbf{X}^1)=\mathbf{1}_{\{\tau_1^1-\tau_0^1\leq T_N\}}\ldots \,\mathbf{1}_{\{\tau_N^1-\tau_{N-1}^1\leq T_1\}}\,f_{N+1}(\mathbf X_0^1,\ldots,\mathbf{X}_{\tau_N^1}^1)
\]
Using the strong Markov property at time $\tau_1$, see Proposition~\ref{prop:propMarkov}, we deduce that $\mathbf P_x^1$-almost surely,
\begin{align*}
\mathbf E_x^1\big[H_{N+1}(\mathbf{X}^1)\,|\,\mathcal{F}_{\tau_1^1}^1\big]&=\mathbf{1}_{\{\tau_1^1-\tau_0^1\leq T_{N+1}\}}\,\mathbf E_x^1\big[\tilde H_{N}(\theta_{\tau_1^1}\mathbf{X}^1)\,\big|\,\mathcal{F}_{\tau_1^1}\big]\\
&=\mathbf{1}_{\{\tau_1^1-\tau_0^1\leq T_{N+1}\}}\,\mathbf E_{\mathbf{X}_{\tau_1^1}^1}^1\big[\tilde H_N(\mathbf{X}^1)\big].
\end{align*}
Then, using the law of $\tau_1^1$, we deduce that $\esp_x^1\big[H_{N+1}(\mathbf{X}^1)\big]$ equals
\[
\mathbf E_x^1\Big[\int_0^{T_{N+1}}\deriv t_{N+1}\,G'\Big(\int_0^{t_{N+1}}b(X_{s}^\ori)\,\deriv s\Big)\,b(X_{t_{N+1}}^\ori)\int_E\pi(X_{t_{N+1}}^\ori,\deriv y_{N+1})\,\esp_{y_{N+1}}^1\big[\tilde H_N(\mathbf{X}^1)\big]\Big].
\]
So by the induction assumption, applied to $f_N(y_N,\ldots,y_1)=f_{N+1}(y_{N+1},y_N,\ldots,y_1)$, concludes the proof.
\end{proof}

\begin{proposition}
\label{prop:coupleN}
Let $f:\mathbb R_+^N\times E^{N}\longrightarrow\mathbb R$ a bounded $\mathcal B(\mathbb R_+^N)\otimes \mathcal{E}^{\otimes N}$-measurable function, then for all $n,N\geq 1$ we have
\begin{align*}
    \esp_x^1\big[f(\tau_1^1,&\ldots,\tau_N^1,\mathbf X_{\tau_1^1}^1,\ldots,\mathbf X_{\tau_N^1}^1)\big]\\
    &=\mathbf E_x^\ori\Big[\int_0^\infty\deriv t_1\,G'\Big(\int_0^{t_1}b(X_{s}^\ori)\,\deriv s\Big)\,b(X_{t_1}^\ori)\int_E\pi(X_{t_1}^\ori,\deriv y_1)\ldots\mathbf E_{y_{N-1}}^\ori\Big[\int_0^\infty\,\deriv t_N\\
    &\times G'\Big(\int_0^{t_N}b(X_{s}^\ori)\,\deriv s\Big)\,b(X_{t_N}^\ori)\int_E\pi(X_{t_N}^\ori,\deriv y_N)\,f(t_1,\ldots,t_1+\ldots+t_N,y_1,\ldots,y_N)\Big]\ldots\Big].
\end{align*}
\end{proposition}

\begin{proof}
We prove the result by induction on $N\geq1$. For $N=1$, let $f:\mathbb R\times E\longrightarrow\mathbb R$ a bounded $\mathcal E^{\otimes2}$-measurable function, by definition of the process $\mathbf X^1$, we have
\begin{align*}
\esp_x^1\big[f(\tau_1^1,\mathbf X_{\tau_1^1}^1)\big]&=\mathbf E_x^\ori\Big[\int_0^\infty\deriv t\,G'\Big(\int_0^tb(X_s^\ori)\,\deriv s\Big)\,b(X_t^\ori)\int_E\pi(X_t^\ori,\deriv y)\,f(t,y)\Big].
\end{align*}
Now, let $f:\mathbb R^{N+1}\times E^{N+1}\longrightarrow\mathbb R$ a bounded $\mathcal B(\mathbb R^{N+1})\otimes E^{N+1}$-measurable function, by the strong Markov property at time $\tau_1$
\[
    \esp_x^1\big[f(\tau_1^1,\ldots,\tau_{N+1}^1,\mathbf X_{\tau_1^1}^1,\ldots,\mathbf X_{\tau_{N+1}^1}^1)\big]=\esp_x^1\big[\psi(\tau_1^1,\mathbf X_{\tau_1^1}^1)\big],
\]
where, for all $s\geq0$ and $x\in E$, $\psi(s,x):=\esp_x^1[f(s,s+\tau_1^1,\ldots,s+\tau_{N}^1,x,\mathbf X_{\tau_1^1}^1,\ldots,\mathbf X_{\tau_{N}^1}^1)]$. By the induction hypotheses applied to the function 
\[
(t_1,\ldots,t_N,x_1,\ldots,x_N)\in E^{2N}\mapsto f(\,\cdot\,,\,\cdot\,+t_1,\ldots,\,\cdot\,+t_N,\,\cdot\,,x_1,\ldots\,x_N)
\]
we have, for all $s\geq0$ and $x\in E$,
\begin{align*}
    \psi(s,x)&=\mathbf E_x^\ori\Big[\int_0^\infty\deriv t_2\,G'\Big(\int_0^{t_2}b(X_{s}^\ori)\,\deriv s\Big)\,b(X_{t_2}^\ori)\int_E\pi(X_{t_2}^\ori,\deriv y_2)\ldots\mathbf E_{y_{N}}^\ori\Big[\\
    &\times\int_0^\infty\,\deriv t_{N+1} G'\Big(\int_0^{t_{N+1}}b(X_{s}^\ori)\,\deriv s\Big)\,b(X_{t_{N+1}}^\ori)\int_E\pi(X_{t_{N+1}}^\ori,\deriv y_{N+1})\\
    &\times f(s,s+t_2,\ldots,s+t_2+\ldots+t_N,x,y_2,\ldots,y_{N+1})\Big]\ldots\Big].
\end{align*}
So, again by definition of the process $\mathbf X^1$, we get
\[
    \esp_x^1\big[\psi(\tau_1^1,\mathbf X_{\tau_1^1}^1)\big]=\mathbf E_x^\ori\Big[\int_0^\infty\,\deriv t_1\,G'\Big(\int_0^{t_1}b(X_s^\ori)\,\deriv s\Big)\,b(X_s^\ori)\int_E\pi(X_{t_1}^\ori,\deriv y_1)\,\psi(t_1,y_1)\Big],
\]
which conclude the proof of the proposition.
\end{proof}

\section{Main results}
\label{sec:main-results}

\subsection{Accelerated and index processes}

In this section,  we consider, for all $n\geq 1$, a time-accelerated version $\mathbf{X}^n$ of the process $\mathbf{X}^1$. 
More precisely, consider  the time-accelerated process  $(X_t^{\ori,n})_{t\geq 0}:=(X_{nt}^\ori)_{t\geq 0}$ and recall that, for all $i\in I$ and $x_i\in E_i$,
\begin{equation}
	\label{hyp:homogeneity}
	\mathbf{P}_{x_i}^\ori\big(\forall t\geq0,X_t^{\ori,n}\in E_i\big)=\mathbf{P}_{x_i}^\ori\big(\forall t\geq0,X_t^\ori\in E_i\big)=1.
\end{equation}
As in the previous section but replacing $X^\ori$  by $X^{\ori,n}$, we define a process 
\[
\mathbf{X}^n=(\Omega,\mathcal{F}^n,(\mathbf{X}_t^n)_{t\geq0},(\mathbf{P}_x^n)_{x\in E},T_\Delta^n)
\]
on $\Omega$ taking values in $E\cup\{\Delta\}$, which evolves following the dynamics of $X^{\ori,n}$ and transitions at rate $b$ according to the transition kernel $\pi$. 
 We denote by $0<\tau_1^n<\tau_2^n<\cdots$ the transition times of the process $\mathbf X^n$ between the sets $E_i$, $i\in I$, and the lifetime of $\mathbf{X}^n$ is $T_\Delta^n=\lim_{k\to+\infty} \tau_k^n$.

We also define the index process $\mathbf{I}^n=(\mathbf{I}^n_t)_{t\geq0}$ as follows: for all $t\geq0$ and $i\in I$,
\[
\mathbf{I}^n_t=i\:\quad\textup{if and only if}\quad\mathbf{X}^n_t\in E_i.
\]
We observe that $\mathbf{I}^n_t=\phi(\mathbf{X}^n_t)$ for all $t\geq0$.

In the remaining of this paper, we assume the following ergodic property for the process restricted to each $E_i$. Let $\mathcal{A}_i$ the subset of functions $f \in L^1(\mu_i)$ such that, for all $x \in E_i$ and all $t \geq 0$,
\begin{equation}
    \label{eq:finite-integral-trajectory}
    \mathbf{P}_{x_i}^\ori \Big( \int_0^t | f(X_s^\ori) |\, \deriv s < + \infty \Big) = 1.
\end{equation}
\begin{assumption}
\label{assumption:ergo}

For all $i\in I$, there exists a probability distribution $\mu_i$ on $E_i$ such that, for all $h\in \mathcal{A}_i$,
 and $x_i\in E_i$
\begin{equation}
\label{hyp:ergo}
\lim_{t\to\infty}\frac{1}{t}\int_0^t h(X_s^\ori)\,\deriv s=\mu_i(h)\quad\mathbf{P}_{x_i}^\ori\textup{-p.s.}
\end{equation}
where $\mu_i(h):=\int_E h(x)\,\mu_i(\deriv x)$.
\end{assumption}
We also make the following assumption on the transition rate function $b$:
\begin{assumption}
\label{hyp:b-UI}
 We assume that the family
\begin{equation}
	\label{hyp:unifintegr}
	\Big(\frac{1}{n}\int_0^{nt} b(X_s^\ori)\,\deriv s\Big)_{n\geq 1}
\end{equation}
is uniformly integrable for all $t\geq0$ under $\mathbf{P}_{x}^\ori$ for all $x\in E$.
\end{assumption}

\begin{remark}
In the Markov case, Assumption~\ref{assumption:ergo} is equivalent to the positive Harris recurrence of $X^\ori$ on each $E_i$. This rather strong assumption enables to make no assumption on the regularity of the transition rate $b$, as well as to state strong convergence results (ie, for bounded measurable functionals). The non regularity of $b$ can be useful in some applications. For example,  in \cite{BBCF25}, the authors consider an individual based model where the birth rate function has a discontinuity point. Assumptions~\ref{assumption:ergo} and~\ref{hyp:b-UI} are commented in greater detail in Section~\ref{sec:comments-hyp}.
\end{remark}

Our purpose is to prove that the index process $\mathbf I^n$ converges, when $n$ goes to infinity, to a simpler process. At the limit, the position of the process in $E$ is not necessarily well defined, while the index process evolves autonomously: the jump times and post-jump locations of the index process do no longer depends on the position in $E$. We also prove that the sequence of positions of the process in $E$ at  jump times converges to a Markov chain on $E$.

To introduce the limiting index process and limiting Markov chain, we first recall the definition of $\phi$, the projection of $E$ on $I$ i.e. for all $x\in E$, $\phi(x)=i$ if and only if $x\in E_i$. We also define a transition kernel $\nu$  from $E$ to $I$ as  the pushforward of $\pi$ by $\phi$:  for all $f:I\to[0,\infty]$ a $\mathcal I$-measurable function we have
\[
    \int_I f(i)\,\nu(x,\deriv i)=\int_Ef\circ\phi(y)\,\pi(x, \deriv y).
\]
Note that for all $A \in \mathcal{I}$ and $x \in E$,
\[
\nu(x, A) = \pi\Big(x, \bigcup_{i \in A} E_i\Big)
\]
\noindent \textbf{Definition of the limiting index process.} 
Let $(\Omega,\mathcal F, (\mathbf I_t)_{t\geq 0},(\mathbf P_i)_{i\in I})$ be a semi-Markov process whose dynamics are defined as follows:  for all $i\in I$, under $\mathbf P_i$  $(\mathbf I_t)_{t\geq 0}$ is an $I$-valued pure jump semi-Markov process 
with $\mathbf P_i(\mathbf I_0=i)=1$, such that its  first jump time $\tau_1$ satisfies
\[
\mathbf P_i(\tau_1\leq t)=G(t\mu_i(b)),\ \forall t\geq 0,
\]
and such that its jumps transitions probabilities $(P(i,j))_{i,j\in I}$ are given by
\begin{equation}
\label{eq:jumpMarkov}
P(i,J)=\int_E\frac{\mu_i(\deriv x)}{\mu_i(b)}\,b(x)\,\nu(x,J), \quad \forall J \in\mathcal{I}.
\end{equation}
We note $(\tau_k)_{k\geq1}$ the sequence of jump times of $\mathbf I$. We define the explosion time of $\mathbf I$ as
\[
    \tau_\infty=\lim_{k\to\infty}\tau_k\in[0,+\infty],
\]
and we set $\mathbf I_t=\Delta_{I}$ $\forall t\geq \tau_\infty$, where $\Delta_{I}\notin I$ is a cemetery point. 
 
\medskip 
\noindent \textbf{Definition of the limiting Markov chain on $E$.} 
Let $(Y_n)_{n \geq 0}$ the Markov chain on $E\cup \{\Delta\}$, whose transition kernel $Q$ is defined as follows: for all $x\in E\cup\{\Delta\}$ and bounded measurable function $f:E \cup \{\Delta\} \to\mathbb R$,
\[
Qf(x) = \begin{cases}
\int_E\frac{\mu_{\phi(x)}(\deriv y)}{\mu_{\phi(x)}(b)}\,b(y)\,\pi(y,f),\quad \text{ if }\mu_{\phi(x)}(b)>0, \\
f(\Delta),\quad \text{ if }\mu_{\phi(x)}(b)=0,
\end{cases}
\]
where $b(\Delta):=0$.
Informally, starting from $Y_0 =x$, with probability $\frac{\mu_{\phi(x)}(b \,\cdot)}{\mu_{\phi(x)}}$, one chooses a position, say $Y_1^-$, before the jump, and then picks $Y_1$ according to $\pi(Y_1^-,\,\cdot\,)$.

Note that, for all $x\in E$, $Q f(x)$ depends on $x$ only through $\phi(x)$, so that the law of $Y_{n+1}$ depends on $Y_n$ only through $\phi(Y_n)$.
We also observe that
\[
Pf(\phi(x)) = Q (f \circ \phi)(x)
\]
so that $\mathbf I_{\tau_n} = \phi(\mathbf{Y}_{n})$, for all $n\geq 0$.

\subsection{Convergence  of the index process}

In this section, we first state the convergence of $\mathbf X^n$ when evaluated at the first jump time. We then consider the convergence of the process evaluated at successive jump times. The following result is proved in Section~\ref{sec:proofprop:convergence1}.

\begin{proposition}
\label{prop:convergence1}
Fix $T>0$ and let  $f:E\longrightarrow\mathbb{R}$ be a bounded $\mathcal{E}$-measurable function. 
Then, for all $i\in I$, $x_i\in E_i$
\[
\lim_{n\to\infty}\espn_{x_i}\big[f(\mathbf{X}_{\tau_1^n}^n)\,\mathbf{1}_{\{\tau_1^n\leq T\}}\big]=\Big(\int_0^T\mu_i(b)\,G'(t\mu_i(b))\,\deriv t\Big)\times\Big(\int_E\frac{\mu_i(\deriv x)}{\mu_i(b)}\,b(x)\,\pi(x,f)\Big).
\]
\end{proposition}

\begin{remark}
\label{rem:oscillation}
Let us observe that we do not have, in general, convergence of $X_t^n$ for a fixe $t>0$, since our assumptions allow for periodic behaviors.  Indeed, consider the simple example where $I=\{1,2\}$, $E_1=[-1,1]$ and $E_2=[2,4]$, $X^\ori_t:=\cos(X^\ori_0+t)$ on $E_1$ and $X^\ori_t:=3+\cos(X^\ori_0+t)$, $b\equiv 1$ and $\pi(x,\cdot)=\mathbf 1_{x\in E_1}\delta_3+\mathbf 1_{x\in E_2}\delta_0$. If $\mathbf X^n_0=0$, then, at any time $t\geq 0$, there is a probability $e^{-t}$ that the process has not jumped before time $t$, and hence that $\mathbf X^n_t=\cos(nt)$. Hence with probability at least $e^{-t}$ the position $\mathbf X^n_t$ oscilates between $-1$ and $1$ as $n\to \infty$ and does not converge.
\end{remark}

\begin{remark}
We note that $X^n_{\tau_1^n}$ converges in law to a random variable distributed according to $\mu(b\,\cdot\,)/\mu(b)$, which is in general not equal to $\mu$. Moreover, $(\tau_1^n,X_{\tau_1^n})$ converges in law to a couple of independent variables.
\end{remark}

\begin{remark}
The proof of Proposition~\ref{prop:convergence1} cannot be generalized to the case where $G$ is not absolutely continuous, since Proposition~\ref{prop:moyenne1} requires the absolute continuty of $G$. Actually, there are examples where $G$ is not absolutely continuous and the convergence result of Proposition~\ref{prop:convergence1} does not hold true. 
For instance, if $b\equiv 1$ and $G(x)=\mathbf 1_{x\geq 1}$, then $\tau_1^n=1$ almost surely, and one faces the same issue as in Remark~\ref{rem:oscillation}
\end{remark}

We are now in a position to state one of our main result, proved in Section~\ref{sec:proofthm:cv-I-jump_times}.
\begin{theorem}
\label{thm:cv-X-jump_times}
Let $x\in E$, $N\geq0$, $T_N>\ldots>T_1>0$ positive real numbers and $f: E^N\rightarrow\mathbb{R}$ a bounded measurable function, then
\begin{align*}
	\espn_{x}\big[f(\mathbf{X}_{\tau_1^n}^n,\ldots,\mathbf{X}_{\tau_N^n}^n)\,\mathbf{1}_{\{\tau_1^n\leq T_1\}}\ldots&\mathbf{1}_{\{\tau_N^n-\tau_{N-1}^n\leq T_N\}}\big] \\
	&\underset{n\to\infty}{\longrightarrow}\mathbf E_{x}\big[f\big(\mathbf{Y}_{1},\ldots,\mathbf{Y}_{N}\big)\,\mathbf{1}_{\{\tau_1\leq T_1\}}\ldots\mathbf{1}_{\{\tau_{N}-\tau_{N-1}\leq T_N\}}\big].
\end{align*}
\end{theorem}
By applying the previous result to $f \circ \phi$ where $f : E^N \to \mathbb R$ is a bounded measurable function, we obtain the following corollary:
\begin{corollary}
\label{cor:cv-I-jump_times}
Let $i\in I$, $x_i\in E_i$, $N\geq0$, $T_N>\ldots>T_1>0$ positive real numbers and $f: I^N\rightarrow\mathbb{R}$ a bounded measurable function, then
\begin{align*}
\mathbf E_{x_i}^n\big[f\big(\mathbf{I}_{\tau_1^n}^n,\ldots,\mathbf{I}_{\tau_N^n}^n\big)\,\mathbf{1}_{\{\tau_1^n\leq T_1\}}\ldots&\mathbf{1}_{\{\tau_{N}^n-\tau_{N-1}^n\leq T_N\}}\big]\\
&\underset{n\to\infty}{\longrightarrow}\mathbf E_{i}\big[f\big(\mathbf I_{\tau_1},\ldots,\mathbf I_{\tau_N}\big)\,\mathbf{1}_{\{\tau_1\leq T_1\}}\ldots\mathbf{1}_{\{\tau_{N}-\tau_{N-1}\leq T_N\}}\big].
\end{align*}
\end{corollary}

\begin{remark}
Using Corollary~\ref{cor:cv-I-jump_times}, one easily checks that the convergence holds true for other functions, e.g. with $\mathbf{1}_{\{\tau_1^n\leq T_1\}}\ldots\mathbf{1}_{\{\tau_{N}^n\leq T_N\}}$ or with finite variation function of the times instead of $\mathbf{1}_{\{\tau_1^n\leq T_1\}}\ldots\mathbf{1}_{\{\tau_{N}^n-\tau_{N-1}^n\leq T_N\}}$, and, for appropriate state spaces, with regular functions of $\tau_1^n,\ldots,\tau_N^n$, $\mathbf{I}_{\tau_1^n}^n,\ldots,\mathbf{I}_{\tau_N^n}^n$ instead of $f\big(\mathbf{I}_{\tau_1^n}^n,\ldots,\mathbf{I}_{\tau_N^n}^n\big)\,\mathbf{1}_{\{\tau_1^n\leq T_1\}}\ldots\mathbf{1}_{\{\tau_{N}^n-\tau_{N-1}^n\leq T_N\}}$.
\end{remark}

\begin{remark}
Actually, the result still holds true for functions of $X_{\tau_1^n}^n,\ldots,X_{\tau_N^n}^n$, but the limiting expectation is quite involved. We refer the reader to Theorem~~\ref{thm:cv-X-jump_times} in Section~\ref{sec:proofthm:cv-I-jump_times}. 
\end{remark}

We consider now the question of convergence at fixed times. While the result does not hold true when considering the convergence of $\mathbf X_T^n$ (as explained in Remark~\ref{rem:oscillation}), it does hold true for the index process.

\begin{corollary}
\label{thm:cv-I-fixed-time}
Let $N\geq 1$ a positive integer, $T_N>\ldots>T_1>0$ strictly positive real numbers, $f:I^N\rightarrow\mathbb{R}$ a bounded measurable function and $x_i\in E_i$ where $i\in I$, then
\[
\mathbf E_{x_i}^n[f(\mathbf{I}_{T_1}^n,\ldots,\mathbf{I}_{T_N}^n)\,\mathbf 1_{\{T_N\leq\tau_{k}^n\}}]\underset{n\to\infty}{\longrightarrow}\mathbf E_i[f(\mathbf I_{T_1},\ldots,\mathbf I_{T_N})\,\mathbf 1_{\{T_N\leq\tau_{k}\}}],\quad\forall k\geq1.
\]
\end{corollary}
\begin{proof}
Let $k, N\geq 1$ positive integers, $T_{N}>\ldots>T_1>0$ strictly positive real numbers, $f:I^{N}\rightarrow\mathbb{R}$ a bounded measurable function and $x_i\in E_i$ where $i\in I$, we have 
\begin{align*}
    \espn_{x_i}\big[f\big(\mathbf{I}_{T_1}^n,\ldots,\mathbf{I}_{T_{N}}^n\big)\,\mathbf 1_{\{T_N\leq\tau_{k}^n\}}\big]&=\sum_{j_1=0}^{k-1}\espn_{x_i}\big[f\big(\mathbf I_{T_1}^n,\ldots,\mathbf I_{T_{N}}^n\big)\,\mathbf 1_{\{\tau_{j_1}^n\leq T_1<\tau_{j_1+1}^n\}}\,\mathbf 1_{\{T_N\leq\tau_{k}^n\}}\big]\\
    &=\sum_{j_1=0}^{k-1}\espn_{x_i}\big[f\big(\mathbf I_{\tau_{j_1}^n}^n,\ldots,\mathbf I_{T_{N}}^n\big)\,\mathbf 1_{\{\tau_{j_1}^n\leq T_1<\tau_{j_1+1}^n\}}\,\mathbf 1_{\{T_N\leq\tau_{k}^n\}}\big]
\end{align*}
and so on, we obtain for the previous term
\[
    \sum_{0\leq j_1\leq\ldots\leq j_N\leq k-1}\espn_{x_i}\big[f\big(\mathbf I_{\tau_{j_1}^n}^n,\ldots,\mathbf I_{\tau_{j_N}^n}^n\big)\,\mathbf 1_{\{\tau_{j_1}^n\leq T_1<\tau_{j_1+1}^n\}}\ldots\mathbf 1_{\{\tau_{j_N}^n\leq T_N<\tau_{j_{N}+1}^n\}}\big]
\]
By Corollary~\ref{cor:cv-I-jump_times} when $n$ goes to infinity, the previous term goes to
\[
    \sum_{0\leq j_1\leq\ldots\leq j_N\leq k-1}\mathbf E_i\big[f(\mathbf I_{\tau_{j_1}},\ldots,\mathbf I_{\tau_{j_1}})\,\mathbf 1_{\{\tau_{j_1}\leq T_1<\tau_{j_1+1}\}}\ldots\mathbf 1_{\{\tau_{j_N}\leq T_N<\tau_{j_{N}+1}\}}\big]
\]
which is equal to
\[
    \mathbf E_i\big[f\big(\mathbf{I}_{T_1},\ldots,\mathbf{I}_{T_{N}}\big)\,\mathbf 1_{\{T_N\leq\tau_{k}\}}\big].
\]
So, we have shown that
\[
    \mathbf E_{x_i}^n[f(\mathbf{I}_{T_1}^n,\ldots,\mathbf{I}_{T_N}^n)\,\mathbf 1_{\{T_N\leq\tau_{k}^n\}}]\underset{n\to\infty}{\longrightarrow}\mathbf E_i[f(\mathbf I_{T_1},\ldots,\mathbf I_{T_N})\,\mathbf 1_{\{T_N\leq\tau_{k}\}}],\quad\forall k\geq1
\]
which concludes the proof of the corollary.
\end{proof}

\subsection{Explosion/non-explosion properties}
Since the processes considered in the previous section are defined iteratively, they are only defined up to their explosion time $\tau_\infty^n:=\lim_{k\to\infty} \tau_k^n$. Similarly, the process $\mathbf I$ is defined up to its explosion time $\tau_\infty:=\lim_{k\to\infty} \tau_k$. In this section, we investigate the relation between the explosion time of the accelerated processes and of the limiting process.

\begin{proposition}
\label{prop:explosion}
Let $T>0$ a strictly positive real number, $f:I\to\mathbb R$ a non-negative bounded measurable function and $x_i\in E_i$ where $i\in I$.  Then
\begin{equation}
\label{eq:exploineq}
    \liminf_{n\to\infty}\mathbf E_{x_i}\big[f\big(\mathbf I_T^n\big)\,\mathbf 1_{\{T<\tau_\infty^n\}}\big]\geq\mathbf E_i\big[f\big(\mathbf I_T\big)\,\mathbf 1_{\{T<\tau_\infty\}}\big].
\end{equation}
If in addition $\mathbf I$ is non-explosive under $\mathbf P_i$, i.e. $\mathbf P_i(\tau_\infty=\infty)=1$, then
\[
    \espn_{x_i}\big[f\big(\mathbf I_T^n\big)\,\mathbf 1_{\{T<\tau_\infty^n\}}\big]\underset{n\to\infty}{\longrightarrow}\mathbf E_i\big[f\big(\mathbf I_T\big)\big].
\]
\end{proposition}
\begin{proof}
Let $f:I\longrightarrow\mathbb R$ a  bounded $\mathcal I$-measurable function and $T>0$ a strictly positive real number. Without loss of generality we can assume that $f$ is positive. For all $k\geq0$, $i\in I$ and $x_i\in E_i$ we have, according to Corollary~\ref{thm:cv-I-fixed-time},
\[
\espn_{x_i}\big[f\big(\mathbf I_T^n\big)\,\mathbf 1_{\{\tau_k^n>T\}}\big]\underset{n\to\infty}{\longrightarrow}\mathbf E_i\big[f\big(\mathbf I_T\big)\,\mathbf 1_{\{\tau_k>T\}}\big].
\]
Reminding that for all $n\geq1$, the jumps sequence $(\tau_k^n)_{k\geq0}$ is increasing, by the monotone convergence theorem we have both
\[
\espn_{x_i}\big[f\big(\mathbf I_T^n\big)\,\mathbf 1_{\{\tau_k^n>T\}}\big]\underset{k\to\infty}{\longrightarrow}\espn_{x_i}\big[f\big(\mathbf I_T^n\big)\,\mathbf 1_{\{\tau_\infty^n>T\}}\big]\quad\forall n\geq1,
\]
and 
\[
\mathbf E_i\big[f\big(\mathbf I_T\big)\,\mathbf 1_{\{\tau_k>T\}}\big]\underset{k\to\infty}{\longrightarrow}\mathbf E_i\big[f\big(\mathbf I_T\big)\,\mathbf 1_{\{\tau_\infty>T\}}\big].
\]
Then, for all $\varepsilon>0$, there exists $k_\varepsilon\geq0$ such that
\[
\mathbf E_i\big[f\big(\mathbf I_T\big)\,\mathbf 1_{\{\tau_k>T\}}\big]\geq\mathbf E_i\big[f\big(\mathbf I_T\big)\,\mathbf 1_{\{\tau_\infty>T\}}\big]-\varepsilon/2\quad\forall k\geq k_\varepsilon,
\]
and 
\begin{align*}
\espn_{x_i}\big[f\big(\mathbf I_T^n\big)\,\mathbf 1_{\{\tau_{k_\varepsilon}^n>T\}}\big]&\geq\mathbf E_i\big[f\big(\mathbf I_T\big)\,\mathbf 1_{\{\tau_{k_\varepsilon}>T\}}\big]-\varepsilon/2\\
&\geq\mathbf E_i\big[f\big(\mathbf I_T\big)\,\mathbf 1_{\{\tau_\infty>T\}}\big]-\varepsilon.
\end{align*}
Using again that $(\tau_k^n)_{k\geq0}$ is increasing for all $n\geq1$ we have
\begin{align*}
\espn_{x_i}\big[f\big(\mathbf I_T^n\big)\,\mathbf 1_{\{\tau_{\infty}^n>T\}}\big]&\geq\espn_{x_i}\big[f\big(\mathbf I_T^n\big)\,\mathbf 1_{\{\tau_{k_\varepsilon}^n>T\}}\big]\\
&\geq\mathbf E_i\big[f\big(\mathbf I_T\big)\,\mathbf 1_{\{\tau_\infty>T\}}\big]-\varepsilon,
\end{align*}
so
\[
\liminf_{n\to\infty}\espn_{x_i}\big[f\big(\mathbf I_T^n\big)\,\mathbf 1_{\{\tau_{\infty}^n>T\}}\big]\geq\mathbf E_i\big[f\big(\mathbf I_T\big)\,\mathbf 1_{\{\tau_\infty>T\}}\big],
\]
which conclude the proof of~\eqref{eq:exploineq}.

Let us now consider the case where $\mathbf I$ is non-explosive, i.e. $\mathbb P_i(\tau_\infty=\infty)=1$. Let $\varepsilon>0$, there exists $k_\varepsilon$ such that for all $k\geq k_\varepsilon$
\begin{equation}
\label{ineq:proofexplosion}
\Big\lvert\mathbf E_i\big[f(\mathbf I_T)\,\mathbf 1_{\{\tau_k>T\}}-\mathbf E_i\big[f(\mathbf I_T)\,\mathbf 1_{\{\tau_\infty>T\}}\big]\Big\rvert\leq\varepsilon.
\end{equation}
We decompose 
\begin{align*}
\espn_{x_i}\big[f(\mathbf I_T^n)\,\mathbf 1_{\{\tau_\infty^n>T\}}\big]&=\espn_{x_i}\big[f(\mathbf I_T^n)\,\mathbf 1_{\{\tau_k^n>T\}}\big]+\espn_{x_i}\big[f(\mathbf I_T^n)\,\mathbf 1_{\{\tau_\infty^n>T\geq\tau_k^n\}}\big]\\
&\leq\espn_{x_i}\big[f(\mathbf I_T^n)\,\mathbf 1_{\{\tau_k^n>T\}}\big]+\|f\|_\infty\,\mathbf P_{x_i}\big(\tau_k^n\leq T\big).
\end{align*}
Since
\[
\mathbf P_{x_i}^n\big(\tau_k^n>T\big)\underset{n\to\infty}{\longrightarrow}\mathbf P_i\big(\tau_k>T\big)\underset{k\to\infty}{\longrightarrow}1,
\]
increasing $k_\varepsilon$ if necessary, there exists $n_\varepsilon$ such that
\[
\mathbf P_{x_i}^n\big(\tau_{k_\varepsilon}^n>T\big)\geq 1-\varepsilon\quad\forall n\geq n_\varepsilon.
\]
Hence
\[
\Big\lvert\espn_{x_i}\big[f(\mathbf I_T^n)\,\mathbf 1_{\{\tau_\infty^n>T\}}-\espn_{x_i}\big[f(\mathbf I_T^n)\,\mathbf 1_{\{\tau_{k_\varepsilon}^n>T\}}\big]\Big\rvert\leq\|f\|_\infty\,\mathbf P_{x_i}^n\big(\tau_{k_\varepsilon}^n\leq T\big)\leq\|f\|_\infty\,\varepsilon.
\]
and, according to Theorem~\ref{thm:cv-I-fixed-time} and up to a change of $n_\epsilon$,
\[
\Big|\espn_{x_i}\big[f(\mathbf I_T^n)\,\mathbf 1_{\{\tau_{k_\varepsilon}^n>T\}}\big]-\mathbf E_i\big[f(\mathbf I_T)\,\mathbf 1_{\{\tau_{k_\varepsilon}>T\}}\big]\Big|\leq \varepsilon,\ \forall n\geq n_\varepsilon.
\]
By addition and substraction,
\begin{align*}
\Big\lvert\espn_{x_i}\big[f(\mathbf I_T^n)\,\mathbf 1_{\{\tau_\infty^n>T\}}\big]-\mathbf E_i\big[f(\mathbf I_T)\,\mathbf 1_{\{\tau_\infty>T\}}\big]\Big\rvert&\leq\Big\lvert\espn_{x_i}\big[f(\mathbf I_T^n)\,\mathbf 1_{\{\tau_\infty^n>T\}}\big]-\espn_{x_i}\big[f(\mathbf I_T^n)\,\mathbf 1_{\tau_{k_\varepsilon}^n>T\}}\big]\Big\rvert\\
&+\Big\lvert\espn_{x_i}\big[f(\mathbf I_T^n)\,\mathbf 1_{\{\tau_{k_\varepsilon}^n>T\}}\big]-\mathbf E_i\big[f(\mathbf I_T)\,\mathbf 1_{\{\tau_k>T\}}\big]\Big\rvert\\
&+\Big\lvert\espn_{x_i}\big[f(\mathbf I_T)\,\mathbf 1_{\{\tau_{k}>T\}}\big]-\mathbf E_i\big[f(\mathbf I_T)\,\mathbf 1_{\{\tau_\infty>T\}}\big]\Big\rvert
\end{align*}
so, using \eqref{ineq:proofexplosion}, we have
\[]
\limsup_{n\to\infty}\Big\lvert\espn_{x_i}\big[f(\mathbf I_T^n)\,\mathbf 1_{\{\tau_\infty^n>T\}}\big]-\mathbf E_i\big[f(\mathbf I_T)\,\mathbf 1_{\{\tau_\infty>T\}}\big]\Big\rvert\leq\big(2+\|f\|_\infty\big)\,\varepsilon.
\]
Finally 
\[]
\espn_{x_i}\big[f(\mathbf I_T^n)\,\mathbf 1_{\{\tau_\infty^n>T\}}\big]\underset{n\to\infty}{\longrightarrow}\mathbf E_i\big[f(\mathbf I_T)\,\mathbf 1_{\{\tau_\infty>T\}}\big]=\mathbf E_i\big[f(\mathbf I_T)\big].
\]
\end{proof}

\begin{example}
\label{exa:31}
While there are obviously situations where the inequality of the previous result is an equality, this is not, in general, the case. To illustrate this, let us build an example where the inequality is strict for $f\equiv 1$. We consider the situation (see Figure~\ref{fig:exa31} below) where $I=\mathbb N$, $E_i=\{(i,0),(i,1)\}$ for all $i\in \mathbb N$, $b(i,0)=0$ and $b(i,1)=i^2$. In addition we assume that, for all $i\in \mathbb N$, $X^\ori$ starting in $E_i$ is a continuous time process with transition rate $1$ from $(i,0)$ to $(i,1)$ and transition rate $0$ from $(i,1)$ to $(i,0)$. We also consider the transition kernel 
\[
    \pi\big((i,x),(j,y)\big)=\begin{cases}
    1&\text{if }j=i+1\text{ and }y=0,\\
    0&\text{otherwise.}
    \end{cases}
\]
\begin{figure}[H]
  \centering
\begin{tikzpicture}[>=stealth, node distance=3cm and 2.5cm, auto,
  state/.style={circle, draw, fill=gray!20, minimum size=18mm, inner sep=1pt}
]

\node[state] (i0) at (0,-1) {$(0,0)$};
\node[state] (i1) at (0,2) {$(0,1)$};
\node[state] (j0) at (3,-1) {$(1,0)$};
\node[state] (j1) at (3,2) {$(1,1)$};

\node[below=6pt] at (i0.south) {$E_0$};
\node[below=6pt] at (j0.south) {$E_1$};

\draw[->, thick] 
  (i0) -- node[left] {$n$} 
  node[right] {$X^{\mathrm{o},n}$} 
  (i1);
\draw[->, thick] (j0) -- node[left] {$n$} 
  node[right] {$X^{\mathrm{o},n}$} (j1) ;

\draw[->, thick] (i1) -- node[left] {$1\,$} 
  node[right] {$\,\mathbf{X}^n$} (j0);

\node at (5,0.5) {\large $\cdots$};

\node[state] (ei0) at (7,-1) {$(i,0)$};
\node[state] (ei1) at (7,2) {$(i,1)$};
\node[state] (eip10) at (10,-1) {$(i+1,0)$};
\node[state] (eip11) at (10,2) {$(i+1,1)$};

\draw[->, thick] (ei0) -- node[left] {$n$} 
  node[right] {$X^{\mathrm{o},n}$} (ei1);
\draw[->, thick] (eip10) -- node[left] {$n$} 
  node[right] {$X^{\mathrm{o},n}$} (eip11);
\draw[->, thick] (ei1) -- node[left] {$i^2\,$} 
  node[right] {$\,\mathbf{X}^n$} (eip10);

\node[below=6pt] at (ei0.south) {$E_i$};
\node[below=6pt] at (eip10.south) {$E_{i+1}$};

\node at (12,0.5) {\large $\cdots$};
\end{tikzpicture}
\caption{\label{fig:exa31} Graphical representation of the state space of Example~\ref{exa:31}}
\end{figure}
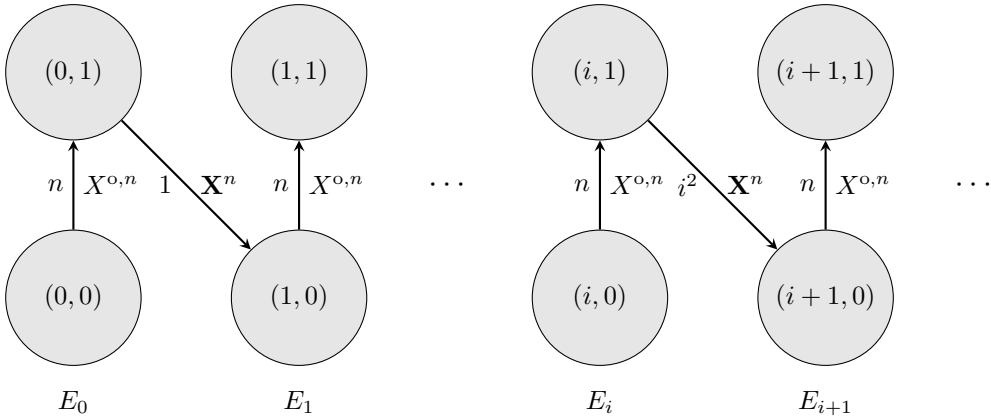

We first observe that, for all $i\in \mathbb N$, the ergodic stationary distribution $\mu_i$ of $X^\ori$ on $E_i$ is such that $\mu_i(i,0)=0$ and $\mu_i(i,1)=1$. As a consequence, the limiting process $\mathbf I$ is a pure jump Markov process on $\mathbb N$ with transition rate $\mu_i(b)=i^2$ from any state $i$ to $i+1$. In particular it is explosive and hence there exists $T>0$ such that 
$\mathbf P(T<\tau_\infty)<1$. However, for any fixed $n$, the process $\mathbf I^n$ is non-explosive (since after each jump from $i$ to $i+1$, it remains an exponential time of parameter $n$ on the state $i$). As a consequence, for all $n\in \mathbb N$, $\mathbf P(T<\tau^n_\infty)=1$. This shows that there is a strict inequality in~\eqref{eq:exploineq}.
\end{example}

\subsection{Convergence in law of the index process  in Skorokhod topology}

In this section, we assume that $I$ is Polish. For all $T > 0$, we let $D([0,T], I)$ denote as usual the space of càdlàg functions defined on $[0,T]$ and with values in $I$. Using the convergence of the finite-dimensional marginals stated in Corollary~\ref{thm:cv-I-fixed-time}, we prove the convergence of $(\mathbf I^n_t)_{t \in[0,T]}$ to $(\mathbf{I}_t)_{t \in[0,T]}$ in law for the Skorokhod topology on $D([0,T], I)$, under mild assumption on the limit process $\mathbf{I}$.

In the following statement, we say that a probability measure $m$ on $I$ is \emph{tight} if, for all $\varepsilon>0$, there exists a compact  $K_\varepsilon\subset I$ such that $m(K)\geq 1-\varepsilon$. In particular, if $I$ is locally compact, all probability measures are tight.
\begin{theorem}
\label{thm:convSkorokhod}
Let $i\in I$ and $T>0$. We assume that $\mathbf I$ is non-explosive up to time $T$ under $\mathbf P_i$, i.e $\mathbf P_i(\tau_\infty>T)=1$, and that, for all $t\in [0,T]$, the law of $\mathbf I_t$ under $\mathbf P_i$ is tight. Then $(\mathbf I^n_t)_{t \in[0,T]}$ converges to $(\mathbf{I}_t)_{t \in[0,T]}$ in law for the Skorokhod topology on $D([0,T], I)$.
\end{theorem}

Before turning to the proof of this result, we recall that if $(\mathbf I^n_t)_{t \in[0,T]}$ converges to $(\mathbf{I}_t)_{t \in[0,T]}$ in law for the Skorokhod topology on $D([0,T], I)$ for all $T>0$, then $(\mathbf I^n_t)_{t \geq 0}$ converges to $(\mathbf{I}_t)_{t \geq 0}$ in law for the Skorokhod topology on $D([0,+\infty), I)$. As a consequence, if $\mathbf I$ is non-explosive, i.e $\mathbf P_i( \tau_{\infty} = + \infty) = 1$ and if  for all $t \geq 0$, the law of $\mathbf I_t$ is tight under $\mathbb P_i$ for all $i \in I$, then  $(\mathbf I^n_t)_{t \geq 0}$ converges to $(\mathbf{I}_t)_{t \geq 0}$ in law for the Skorokhod topology on $D([0,+\infty), I)$.

{\color{black}\begin{proof}[Proof of Theorem~\ref{thm:convSkorokhod}]
According to Corollary~\ref{thm:cv-I-fixed-time}, the finite marginals of the processes $\mathbf I^n$ converge in law, when $n\to+\infty$, to the finite marginals of $\mathbf I$. As a consequence, it is sufficient to prove the tightness of $(\mathbf I^n)_{n\geq 1}$ in the Skorokhod topology on $D([0,T],I)$. In order to so, we make use of Theorem~7.2, p.128 in~\cite{EthierKurtz2005}. This result makes use of the following definition: for all $\delta>0$,
\[
w'(\mathbf I^n,\delta):=\inf_{\{t_i\}}\max_{1\leq i\leq k} w_{\mathbf I^n}[t_{i-1},t_i),
\]
with $w_{\mathbf I^n}[t_{i-1},t_i)$ the modulus of continuity of $I^n$ over $[t_{i-1},t_i)$ and the $\{t_i\}$ range over all partitions of the form $0=t_0<t_1<\cdots <t_{k-1} < T\leq t_k$ with steps  stricktly greater than $\delta$. According to the above cited result, the tightness property is ensured if, for all $\varepsilon>0$ and all $t\in [0,T]$, there exists a compact set $K_{\varepsilon,t}\subset I$ such that
\begin{align}
    \label{eq:compactcontainment}
    \limsup_{n\to+\infty} \mathbf P(\mathbf I^n_t\in K_{\varepsilon,t})\geq 1-\varepsilon
\end{align}
and there exists $\delta>0$ such that
\begin{align}
    \label{eq:propertyiEthierKurtz}
    \limsup_{n\to+\infty} \mathbf P(w'(\mathbb I_n,\delta)>\varepsilon)\leq \varepsilon.
\end{align}

In order to check the first property, choose a compact subset $K_{\varepsilon,t}\subset I$ such that 
\[
\mathbf P(\mathbf I_t\in K_{\varepsilon,t})\geq 1-\varepsilon/2
\]
and $k_{\varepsilon}\geq 1$ 
\[
\mathbf P(T\leq \tau_{k_\varepsilon})\geq 1-\varepsilon/2.
\]
Then, according to Corollary~\ref{thm:cv-I-fixed-time}, 
\[
\lim_{n\to/+\infty} \mathbf P(\mathbf I^n_t\in K_{\varepsilon,t}\text{ and }t\leq \tau_{{k_\varepsilon}}^n)=\mathbf P(\mathbf I_t\in K_{\varepsilon,t}\text{ and }t\leq \tau_{k_\varepsilon})\geq 1-\varepsilon,
\]
which proves~\eqref{eq:compactcontainment}.

Let us now check~\eqref{eq:propertyiEthierKurtz}. By definition of $w'$ and since $\mathbf I^n$ is a pure jump process for any $n\geq 1$, we have
\begin{align*}
\left\{ \omega'(\mathbf I^n, \delta) \geq \varepsilon \right\}
&\subset  
\left\{ \exists i\in \{1,\ldots,k_\varepsilon\} :  
\, \tau_i^n -\tau_{i-1}^n> \delta \right\}\cup\{\tau_{k_\varepsilon}<T\}
\end{align*}
and hence
\begin{align*}
    \mathbf P\left(\omega'(\mathbf I^n, \delta) \geq \varepsilon\right)\leq \mathbf P\left(\exists i\in \{1,\ldots,k_\varepsilon\} :  
\, \tau_i^n -\tau_{i-1}^n> \delta\right)+\varepsilon/2.
\end{align*}
Using Corollary~\ref{cor:cv-I-jump_times}, we deduce that
\begin{align*}
    \limsup_{n\to+\infty} \mathbf P\left(\omega'(\mathbf I^n, \delta) \geq \varepsilon\right)\leq \mathbf P\left(\exists i\in \{1,\ldots,k_\varepsilon\} :  
\, \tau_i-\tau_{i-1}> \delta\right)+\varepsilon/2.
\end{align*}
Since $\int_{0^+} b(X_s^\ori)ds < + \infty$, we have $\tau_i-\tau_{i-1}>0$ almost surely, and hence there exists $\delta>0$ small enough so that $\mathbf P\left(\exists i\in \{1,\ldots,k_\varepsilon\} :  
\, \tau_i-\tau_{i-1}> \delta\right)\leq \varepsilon/2$.
we deduce that~\eqref{eq:propertyiEthierKurtz} holds true which concludes the proof of Theorem~\ref{thm:convSkorokhod}.
\end{proof}}

\begin{remark}
We emphasize that the preceding proof stands correct and may thus be useful in other contexts where one seeks to prove the convergence of pure jump process toward a pure jump process. Indeed, it only uses the pure jump property, the convergence of the jump times and the non-degeneracy of the limit.
\end{remark}

\subsection{Comments on Assumption~\ref{assumption:ergo} and~\ref{hyp:b-UI}: Positive Harris recurrence}
\label{sec:comments-hyp}

First we provide an alternative formulation of Assumption~\ref{hyp:b-UI}, when Assumption~\ref{assumption:ergo} is granted.

\begin{lemma}
\label{lemma:assumptionErgo}
If Assumption~\ref{assumption:ergo} holds true, then Assumption~\ref{hyp:b-UI} is equivalent to the following property: for all $i \in I$,
\[
    b\in \mathcal{A}_i \text{ and }\lim_{t \to \infty} \mathbf{E}_x^\ori\Big[ \frac{1}{t}\int_0^t b(X_s^\ori) \, \deriv s\Big] = \mu_i(b)
\]
\end{lemma}

The reverse implication is an immediate consequence of Scheffé's lemma. The direct implication is a consequence of the next lemma, where we show that the uniform integrability of the family $(\frac{1}{n}\int_0^n b(X_s^\ori) \, \deriv s)_{n \geq 1}$ implies that the convergence in~\eqref{hyp:ergo} holds true for $b$:
\begin{lemma}
\label{lem:UI-implies-L1mu}
Under  Assumptions~\ref{assumption:ergo} and~\ref{hyp:b-UI}, we have $b \in \mathcal{A}_i$ for all $i \in I$. In particular, for all $x \in E_i$, 
\[
\lim_{t \to \infty} \frac{1}{t} \int_0^t b(X_s^\ori) \deriv s = \mu_i(b), \quad \mathbf{P}_x^\ori - \mbox{almost surely.}
\]
\end{lemma}

\begin{proof}
Let $x \in E$. Since $(\frac{1}{n}\int_0^n b(X_s^\ori) \, \deriv s)_{n \geq 1}$ is uniformly integrable, it is bounded in $L^1( \mathbf{E}_x^\ori)$. This implies that for all $n \geq 1$, $\int_0^n b(X_s^\ori) \, \deriv s < + \infty$ $\mathbf{P}_x^\ori$ - almost surely and therefore for all $t \geq 0$, $\int_0^t b(X_s^\ori) \, \deriv s < + \infty$. It remains to show that $b \in L^1(\mu_i)$. Assume that this is not the case. Then, for all $N \geq 1$, $b \wedge N \in \mathcal{A}_i$ and thus by~\eqref{hyp:ergo}, if $x \in E_i$ we have $\mathbf{P}_x^\ori$ - almost surely
\begin{align*}
    \liminf_{t \to \infty} \frac{1}{t} \int_0^t b(X_s^\ori) \, \deriv s & \geq \lim_{t \to \infty} \frac{1}{t} \int_0^t b(X_s^\ori) \wedge N \, \deriv s \\
    & \geq \mu_i(b \wedge N)
\end{align*}
This last term goes to infinity as $N \to \infty$, yielding
\[
\liminf_{t \to \infty} \frac{1}{t} \int_0^t b(X_s^\ori) \, \deriv s = + \infty
\]
Now, Fatou's lemma implies
\[
+ \infty = \mathbf{E}_x^\ori\Big( \liminf_{n \to \infty}\frac{1}{n} \int_0^n b(X_s^\ori) \, \deriv s \Big) \leq \liminf_{n \to \infty} \mathbf{E}_x^\ori\Big( \frac{1}{n} \int_0^n b(X_s^\ori) \, \deriv s \Big),
\]
which is in contradiction with the fact that  Since $(\frac{1}{n}\int_0^n b(X_s^\ori) \, \deriv s)_{n \geq 1}$ is bounded in $L^1( \mathbf{E}_x^\ori)$. Hence $b \in L^1(\mu_i)$ and the proof is complete.
\end{proof}

Next, we give some practical condition to check Assumption~\ref{assumption:ergo}.  In the case when $X^\ori$ is a Hunt process on a locally compact secondly countable space, then requiring convergence in \eqref{hyp:ergo} for all $h \in \mathcal{A}_i$ and all $x_i \in E_i$ is equivalent to the positive Harris recurrence of $X^\ori$ on $E_i$ (see~\cite{ADR69}). A sufficient, condition for the positive Harris recurrence is the strong ergodicity of $X^\ori$:
\begin{lemma}
\label{lem:strong-ergodic-implies-Harris}
Assume that for all $i \in I$, $X^0$ is a Markov 
process on $E_i$ and that  there exists a distribution $\mu_i$ on $E_i$ such that for all $x \in E_i$,
\[
\lim_{t \to \infty} \| \mathbf{P}_x^\ori( X_t^\ori \in \cdot) - \mu_i \|_{TV} = 0.
\]
Then $X^\ori$ is positive Harris recurrent on $E_i$, and in particular convergence in~\eqref{hyp:ergo} holds for all $h\in \mathcal A_i$.
\end{lemma}

\begin{proof}
Let $h$ be an harmonic function, that is a bounded measurable function $h : E_i \to \mathbb{R}$ such that $P_t h = h$ for all $t \geq 0$, where 
\[
P_t h(x) = \mathbf{E}_x^\ori\big[ h( X_t^\ori)\big].
\]
Then $h(x) = \lim_{t \to \infty}P_th(x) = \mu_i(h)$, which proves that $h$ is actually constant. Now following the proof of Theorem 7.1 in \cite{BH22}, we get that whenever every bounded harmonic function is constant, the convergence in~\eqref{hyp:ergo} holds true for all $h \in \mathcal{A}_i$.
\end{proof}

\begin{remark}
When $X^\ori$ is a $C_b$ - Feller process,  a sufficient condition for positive Harris recurrence (weaker than the one in the previous lemma) is the existence of an accessible weak Doeblin point (see Proposition 4.8 in \cite{B18}).
Here by $C_b$ - Feller, we mean that for all $t \geq 0$, and all bounded continuous function $h$,  $P_t h$ is also continuous, where $P_t h$ is defined in the proof of the lemma above. The existence of a weak Doeblin point can be checked through weak Hörmander condition in the case of Piecewise Deterministic Markov Processes or of diffusions.
\end{remark}
The fact that $b \in L^1(\mu_i)$ does not ensure that the family in \eqref{hyp:unifintegr} is uniformly integrable - not even that it is finite, as illustrated in Example~\ref{example:bad-Harris} below. A sufficient condition is that  $b$ is sufficiently small with respect to a Lyapunov function:

\begin{lemma}
\label{lem:Lyap-implies-UI}
Assume that that there exist functions $W, \tilde W : E \to \mathbb{R}_+$ and $\hat W : E \to \mathbb{R}$ such that
\begin{enumerate}
    \item For all $x \in E$, the process 
    \[
    W(X_t^\ori) - W(x) - \int_0^t \hat W(X_s^\ori) \, \deriv s
    \]
    is a cadlag local martingale under $\mathbf{P}^\ori_x$;
    \item for some constant $C \geq 0$, $\hat W \leq - \tilde W + C$;
    \item $b^{1 + \varepsilon} \leq K \tilde{W}$ for some $\varepsilon, K > 0$.
\end{enumerate}
Then Assumption~\ref{hyp:b-UI} is satisfied.  
\end{lemma}

\begin{proof}
  By points 1. and 2., we have that for all $t \geq 0$, (see e.g. \cite{B18}, Theorem 2.2)
\[
\int_0^t \mathbf{E}^\ori_x \big[ \tilde W(X_s^\ori)\big] \, \deriv s \leq W(x) + C t.
\]
Hence, by point 3. and Jensen inequality, we have
\begin{align*}
    \mathbf{E}_x^\ori \Big[ \Big( \frac{1}{t} \int_0^t b(X_s^\ori) \, \deriv s \Big)^{1 + \varepsilon} \Big] & \leq \mathbf{E}_x^\ori \Big[ \frac{1}{t} \int_0^t b^{1 + \varepsilon} (X_s^\ori) \, \deriv s\Big]\\
    & \leq \frac{1}{t} \int_0^t  \mathbf{E}_x^\ori \big[   \tilde W (X_s^\ori) \, \deriv s\big]\\
    & \leq \frac{W(x)}{t} + C.
\end{align*}
This entails that the family $\big( \frac{1}{n}\int_0^n b(X_s^\ori) \, \deriv s\big)_{n \geq 1}$ is bounded in $L^{1+\varepsilon}(\mathbf{E}_x^\ori)$, hence is uniformly integrable.
\end{proof}
Finally, let us comment on why the convergence in~\eqref{hyp:ergo} is only asked for all functions $h \in \mathcal{A}_i$ and not for all $h \in L^1(\mu_i)$. This is because the set $E_i$ could contains some transient subspace $F_i$, which is therefore of $\mu_i$ measure $0$. Hence, the fact that $h \in L^1(\mu_i)$ does not provide any information on $h(X_s^\ori)$ for starting point $x \in F_i$ and time $s$ smaller than the exit time of $F_i$. The following example illustrate this issue:

\begin{example}
\label{example:bad-Harris}
Consider the set $E = \mathbb{R}_+^* \times \{1,2\}$ and the process $(X,U)$ on $E$, where $U$ is a continuous-time Markov chain on $\{1,2\}$ with transition rates $1$, and $X$ evolves according to
\[
\frac{d X_t}{dt} = F^{U_t}(X_t),
\]
where
\[
F^1(x) = x(1- x), \quad F^2(x) = x(2-x).
\]
It is easily proven that for all starting point $(x,u) \in E$, $X_t$ will enter in finite time in the set $[1,2]$ and then stay there for all subsequent times. The set $F =(0,1) \times \{1,2\} \cup (2, +\infty) \times \{1, 2\}$ is therefore transient. In addition, it can be proven that $(X,U)$ admits a unique stationary $\mu$, necessarily supported by $[1,2] \times \{1, 2\}$. In addition, the law of $(X_t, U_t)$ converges in total variation to $\mu$, which implies that $(X,U)$ is positive Harris recurrent by Lemma~\ref{lem:strong-ergodic-implies-Harris}. Consider now the function $h : E \to \mathbb{R}$ given by $h(x,u) = \frac{1}{2x-1}$ if $x \neq 1/2$ and $h(1/2,u) = 0$. Then $h$ is bounded on $[1,2] \times \{1,2\}$ and therefore integrable with respect to $\mu$. On the other hand, if $x \in (0,1/2)$, then $X$ will grow at least linearly at speed at least $x > 0$ to $1$, so that for all $t$ large enough, 
\[
\int_0^t h(X_s, U_s) \, \deriv s = +\infty,
\]
which implies that $\frac{1}{t} \int_0^t h(X_s) \, \deriv s$ does not converge to $\mu(h)$.
\end{example}

\section{Applications}

\label{sec:appli}

\subsection{Branching processes}

We consider a branching process where each individual carries an internal stochastic trait evolving dynamically in time. More precisely, each individual is endowed with a c\'adl\'ag Markov process $(Y_t)_{t \geq 0}$ on a Polish type space $D$, which describes the evolution of its intrinsic state (for instance, a physiological characteristic, viral charge, or environmental marker). This trait influences both the individual’s lifetime and its reproduction mechanism. We assume that the process $Y$ admits a unique ergodic stationary distribution $\chi$, ensuring that long-time averages along the trajectory of $Y$ are well defined and non-degenerate. In addition, the rate at which an individual reproduces depends on the current state of its trait through a measurable function $\beta : D \to \mathbb R+$, and the offspring distribution is described by a reproduction kernel $K$ from $D$ to the space of offspring traits. The resulting population system is thus an interacting collection of Markovian individuals, each following its own trait dynamics in continuous time, and undergoing branching at state-dependent rates. Such \textit{branching processes} have been extensively studied in the past decades, see e.g.~\cite{harris1963theory,MBP2,horton2023stochastic}.

More formally, we assume that we are given: 
\begin{itemize}
    \item The individual dynamic $(Y_t)_{t \geq 0}$: a Markov process on a measurable space $D$. We assume that $Y$ admits a unique ergodic stationary distribution $\chi$ such that for all $y \in D$, 
\[
\lim_{t \to \infty} \| \mathbf{P}_y( Y_t \in \cdot) - \chi \|_{TV} = 0.
\]
    \item The individual branching rate $\beta : D \to \mathbb{R}_+$: a measurable function such that   for all $y \in D$, the family $(\frac{1}{n}\int_0^n \beta(Y_s) \, \deriv s)_{n \geq 1}$ is uniformly integrable.
    \item The individual reproduction kernel $K$: a transition kernel from $D$ to $\{ \Delta\}\cup \cup_{i\geq 1} D^i$, where $\Delta$ is a cemetery point corresponding to an empty progeny.
\end{itemize}

\begin{remark}
We assume that the law of $Y$ converges in total variation, which is stronger than the positive Harris recurrence, see Lemma~\ref{lem:strong-ergodic-implies-Harris}. This stronger assumption is useful because it can be tensorized: if two independent processes converge in total variation norm, then the couple converges in total variation norm (see Lemma~\ref{lem:tensorisation-cv-VT}), while this is not the case for independent positive Harris recurrent processes. Note that, as proved in Theorem 6.1 in \cite{meyn1993stability}, a continuous time positive Harris recurrent converges in total variation norm if and only if some skeleton chain is irreducible.
\end{remark}

Let $I = \mathbb{N}=\{0,1,\ldots\}$, $E_0 = \{ \Delta \}$, and for $i \geq 1$, $E_i = D^i$ and $E=\cup_{i\in\mathbb N}E_i$. For $x = (x_1, \ldots, x_i) \in E_i$,  let $((Y_t^k)_{t \geq 0})_{1 \leq k \leq i}$ be a family of independent copies of $Y$ such that $Y_0^k = x_k$. Then, we consider a process $X^\ori$ such that 
\[
\mathbf{P}_{(x_1, \ldots, x_i)}^\ori\big(  X_t^\ori = ( Y_t^1, \ldots, Y_t^i), \quad \forall t \geq 0\big) = 1.
\]
We also let $\mu_i = \chi^{\otimes i}$ and $b : E \to \mathbb{R}$ defined by its projection on $E_i$: $b_0=0$ and
\[
b_i(x_1, \ldots, x_i) = \sum_{k=1}^i \beta(x_k).
\]
Finally, set $\pi(\Delta,E)=0$ and
\[
\pi((x_1, \ldots, x_i), dy) = \sum_{k=1}^i \frac{\beta(x_k)}{b_i(x_1, \ldots, x_i)} \big(\delta_{x_1} \otimes \cdots \otimes K(x_k,\, \cdot\,) \otimes \delta_{x_{k+1}} \otimes \cdots \otimes \delta_{x_i}\big)(dy),
\]
and $G(t) = 1 - e^{-t}$. We construct as in Sections~\ref{sec:base} and~\ref{sec:main-results} the process $\mathbf X^n=(\mathbf X_t^n)_{t\geq0}$ by piecing-out with dynamic $X^{\ori,n}$, jump rate $b$ and jump kernel $\pi$. The index process $\mathbf I^n=(\mathbf I_t^n)_{t\geq0}$ by projection of $\mathbf X^n$ on $\mathbb N$. Then, $\mathbf I^n$ is a process taking values in $\mathbb N$, representing the number of particles alive at time $t \geq 0$ (cells, individuals, infective people...). 
\begin{proposition}
\label{prop:branching}
The sequence of processes $(\mathbf{I}^n)_{n\geq 1}$ converges in law for the Skorokhod topology on $D([0,+\infty),\mathbb N)$ to 
a standard continuous-time Galton-Watson tree with reproduction/branching rate $\chi(\beta)$ and offspring distribution \[L = \bigg(\frac{\chi(\beta K( \cdot, D^j))}{\chi(\beta)}\bigg)_{j \in \mathbb N}.\]
 \end{proposition}
  
\begin{proof}

Using  Lemma \ref{lem:tensorisation-cv-VT} and Lemma~\ref{lem:strong-ergodic-implies-Harris} and the hypothesis on the individual dynamic, for all $i\in\mathbb N$, the process $((Y_t^1,\ldots, Y_t^i))_{t\geq0}$ is positive Harris recurrent and for all $f \in \mathcal{A}_i$ and all $x_i \in E_i$
\[
    \frac{1}{t}\int_0^tf(X_s^\ori)\,\deriv s=\frac{1}{t}\int_0^tf(Y_s^1,\ldots,Y_s^i)\,\deriv s\underset{t\to\infty}{\longrightarrow}\mu_i(f)\quad\mathbf P_{x_i}\mathrm{-p.s.}
\]
In addition, by definition of $b$ and the assumption on the individual branching rate $\beta$, the family $(\frac{1}{n}\int_0^n b(X_s^\ori) ds)_{n \geq 1}$ is uniformly integrable. Therefore, Assumptions~\ref{assumption:ergo} and~\ref{hyp:b-UI} are satisfied.

By Theorem~\ref{thm:convSkorokhod},  $\mathbf I^n$  converges in the Skorokhod topology of $D([0,+\infty),\mathbb N)$ to $\mathbf I$, where, for all $i\in\mathbb N$, under $\mathbf P_i$,  $(\mathbf I_t)_{t\geq 0}$ is an $\mathbb N$-valued pure jump Markov process with $\mathbf P_i(\mathbf I_0=i)=1$, such that its  first jump time $\tau_1$ satisfies for all $t\geq0$
\[
\mathbf P_i(\tau_1\leq t)=1-e^{-t\mu_i(b_i)}=1-e^{-t\,i\chi(\beta)}
\]
and such that its jumps transitions probabilities $(P(i,j))_{i\geq 1,j\in\mathbb N}$ are given, for all $j\in\mathbb N$ by
\begin{align*}
\label{eq:jumpMarkov}
P(i,j)&=\int_E\frac{\mu_i(\deriv x)}{\mu_i(b_i)}\,b_i(x)\,\nu(x,j)\\
&=\int_E\frac{\mu_i(\deriv x)}{i\chi(\beta)}\,\sum_{k=1}^i\beta(x_k)\,K(x_k,D^j)\\
&=\frac{\chi\big(\beta\,K(\,\cdot\,,D^j)\big)}{\chi(\beta)}
\end{align*}
where $x=(x_1,\ldots,x_i)\in D^i$. Then $\mathbf I$ is a branching process where each individual branches at rate $\chi(\beta)$, and at a branching event, gives birth and is replaced by $j\geq0$ individuals with probability 
\[
\frac{\chi\big(\beta\,K(\,\cdot\,,D^j)\big)}{\chi(\beta)}.
\]
This concludes the proof of Proposition~\ref{prop:branching}.
\end{proof}
  
  \begin{remark}
  Note that $r_j(x):=\beta(x) K(x, D^j)$ is the rate at which an individual with "characteristic" $x$ is replaced by $j$ individuals. In the limiting process, the rate at which an individual is replaced by $j$ individuals is $\chi(r_j)$.
  \end{remark}
 
 \begin{example}
Consider the case of a binary  branching process, i.e., a branching event correspond to a division or a death. If $r$ and $q$ denotes respectively the division and the death rates, then $\beta = r + q$ is the branching rate. Setting
\[
\bar r = \chi(r), \quad \bar q = \chi(q),
\]
the index process $\mathbf{I}$ is a linear birth and death process with individual birth rate $ \bar r$ and individual death rate $ \bar q$. Hence, it is non-explosive, converges to $0$ is $\bar r \leq \bar q$, while if $\bar r > \bar q$, 
 \[
 \mathbf P\big( \lim_{t \to \infty} \mathbf{I}_t = 0\big) = 1 - \mathbf P\big( \lim_{t \to \infty} \mathbf{I}_t = + \infty \big) = \frac{\bar q}{\bar r}.
 \]
\end{example}

\begin{example}
More generally, let $m = \sum_j j \chi(r_j)$ be the mean number of offsprings. It is well known that, if $m \leq 1$, $\mathbf{I}_t$ converges to $0$ with probability $1$, while if $m > 1$ this probability is strictly less than one, and on the complementary event, it goes to infinity as $t\to+\infty$. This provides information on the limiting behaviour of $\mathbf P(\lim_{t\to+\infty} \mathbf I_t^n=0)$ through the relation
\begin{align*}
    \liminf_{n\to+\infty} \mathbf P\big(\lim_{t\to+\infty} \mathbf I_t^n=0\big)\geq \mathbf P(\lim_{t\to+\infty} \mathbf I_t=0),
\end{align*}
which is obtained similarly as in the proof of Proposition~\ref{prop:explosion}.
However, there is no a priori general systematic relation between $\mathbf P(\liminf_{t\to+\infty} \mathbf I_t^n=+\infty)$ and $\mathbf P(\liminf_{t\to+\infty} \mathbf I_t=+\infty)$. 
\end{example}

\subsection{Generalized contact process}

We consider a contact process with individual stochastic viral load on a graph $G=(E_G,V_G)$, assumed to be of finite maximal degree. Each vertices $k\in V_G$ represents an individual, and we set $i_k=0$ if the individual is susceptible and $i_k=1$ if the individual is infected. Each infected individual has a viral load $x_k\in D$, where $D$ is a Polish set, which influences its rate of healing and its rate of infection of susceptible neighbours. The viral charge of one individual follows a stochastic autonomous dynamic in between infection/healing events. We denote by $I:=\{0,1\}^{(V_G)}$ the space of susceptible/infected configurations, with only finitely many non-zero coordinates.

More precisely, we assume that we are given an initial configuration with only finitely many infected individuals, and
\begin{itemize}
    \item the infection rate $\kappa_{0\to 1}:D\to \mathbb R_+$: an individual $k\in V_G$ who is susceptible (i.e. such that $i_k=0$) is infected independently by each of its infected neighbour $\ell$ at rate $\kappa_{0\to 1}(x_\ell)$; when infected by a neighbour $\ell$, the viral load of $k$ is a random variable distributed according to $p(x_\ell,\cdot)$, where $p$ is a transition kernel from $D$ to $D$; see Figure~\ref{fig:infection};
    \item the healing rate $\kappa_{1\to 0}:D\to\mathbb R_+$: an individual $k\in V_G$ who is infected (i.e. such that $i_k=1$) heals at rate $\kappa_{1\to 0}(x_k)$;
    \item the viral load Markov process $Y$: the viral load $x_k$  of an individual $k\in V_G$ who is infected (i.e. such that $i_k=1$) evolves according to the dynamics of $Y$  independently of the rest of the system up to the healing time of $k$.
\end{itemize}
A similar modelling approach, where each site carries a real-valued process to represent the level of knowledge or information, has been recently introduced in the context of interacting particle systems and knowledge diffusion models in \cite{ContactProcessWithStates}.
\begin{figure}[H]
  \centering
\begin{tikzpicture}[scale=0.84, every node/.style={font=\normalsize}, >=stealth, node distance=3cm and 2.5cm, auto,
  state/.style={circle, draw, fill=gray!20, minimum size=9mm, inner sep=1.5pt}
]

  \foreach \x in {0,1,2}{
    \foreach \y in {0,1,2}{
      \ifnum\x=1
        \ifnum\y=1
          \node[state] (L\x\y) at ({2*\x},{-2*\y}) {0};
        \else
          \ifnum\y=0
            \node[state] (L\x\y) at ({2*\x},{-2*\y}) {$1$};
          \else
            \node[state] (L\x\y) at ({2*\x},{-2*\y}) {$1$};
          \fi
        \fi
      \else
        \node[state] (L\x\y) at ({2*\x},{-2*\y}) {0};
      \fi
    }
  }

  \foreach \x in {0,1}{
    \foreach \y in {0,1,2}{
      \pgfmathtruncatemacro{\xp}{\x+1}
      \draw (L\x\y) -- (L\xp\y);
    }
  }
  \foreach \x in {0,1,2}{
    \foreach \y in {0,1}{
      \pgfmathtruncatemacro{\yp}{\y+1}
      \draw (L\x\y) -- (L\x\yp);
    }
  }

  \foreach \y in {0,1,2}{
    \draw[dotted, thick] (L2\y) -- ++(1,0);
  }
  \foreach \y in {0,1,2}{
    \draw[dotted, thick] (L0\y) -- ++(-1,0);
  }
  \foreach \x in {0,1,2}{
    \draw[dotted, thick] (L\x2) -- ++(0,-1);
  }
  \foreach \x in {0,1,2}{
    \draw[dotted, thick] (L\x0) -- ++(0,1);
  }
  \draw[dotted, thick] (L22) -- ++(1,0);
  \draw[dotted, thick] (L22) -- ++(0,-1);
  \draw[dotted, thick] (L02) -- ++(-1,0);
  \draw[dotted, thick] (L02) -- ++(0,-1);
  \draw[dotted, thick] (L20) -- ++(1,0);
  \draw[dotted, thick] (L20) -- ++(0,1);
  \draw[dotted, thick] (L00) -- ++(-1,0);
  \draw[dotted, thick] (L00) -- ++(0,1);

\draw[->, thick] (6,-2) --  node[above] {$q\Big(Y_t^{(0,1)},\,\cdot\,\Big)+q\Big(Y_t^{(0,-1)},\,\cdot\,\Big)$}node[below] {$\kappa_{0\to1}\Big(Y_t^{(0,1)}\Big)+\kappa_{0\to1}\Big(Y_t^{(0,-1)}\Big)$} (9,-2);

  \foreach \x in {0,1,2}{
    \foreach \y in {0,1,2}{
      \ifnum\x=1
        \ifnum\y=1
          \node[state] (R\x\y) at ({11+2*\x},{-2*\y}) {$1$};
        \else
          \ifnum\y=0
            \node[state] (R\x\y) at ({11+2*\x},{-2*\y}) {$1$};
          \else
            \node[state] (R\x\y) at ({11+2*\x},{-2*\y}) {$1$};
          \fi
        \fi
      \else
        \node[state] (R\x\y) at ({11+2*\x},{-2*\y}) {0};
      \fi
    }
  }

  \foreach \x in {0,1}{
    \foreach \y in {0,1,2}{
      \pgfmathtruncatemacro{\xp}{\x+1}
      \draw (R\x\y) -- (R\xp\y);
    }
  }
  \foreach \x in {0,1,2}{
    \foreach \y in {0,1}{
      \pgfmathtruncatemacro{\yp}{\y+1}
      \draw (R\x\y) -- (R\x\yp);
    }
  }

  \foreach \y in {0,1,2}{
    \draw[dotted, thick] (R2\y) -- ++(1,0);
  }
  \foreach \y in {0,1,2}{
    \draw[dotted, thick] (R0\y) -- ++(-1,0);
  }
  \foreach \x in {0,1,2}{
    \draw[dotted, thick] (R\x2) -- ++(0,-1);
  }
  \foreach \x in {0,1,2}{
    \draw[dotted, thick] (R\x0) -- ++(0,1);
  }
  \draw[dotted, thick] (R22) -- ++(1,0);
  \draw[dotted, thick] (R22) -- ++(0,-1);
  \draw[dotted, thick] (R02) -- ++(-1,0);
  \draw[dotted, thick] (R02) -- ++(0,-1);
  \draw[dotted, thick] (R20) -- ++(1,0);
  \draw[dotted, thick] (R20) -- ++(0,1);
  \draw[dotted, thick] (R00) -- ++(-1,0);
  \draw[dotted, thick] (R00) -- ++(0,1);

\end{tikzpicture}
\caption{\label{fig:infection} Infection rates and distribution of viral load of a newly infected individual, with the notation $q(x_\ell,\cdot)=\frac{\kappa_{0\to 1}(x_\ell)}{\sum_{j\sim k}i_j\,\kappa_{0 \to 1}(x_{j})}\,p(x_\ell,\,\cdot\,)$.  The individual in the center is numbered $(0,0)$.}
\end{figure}
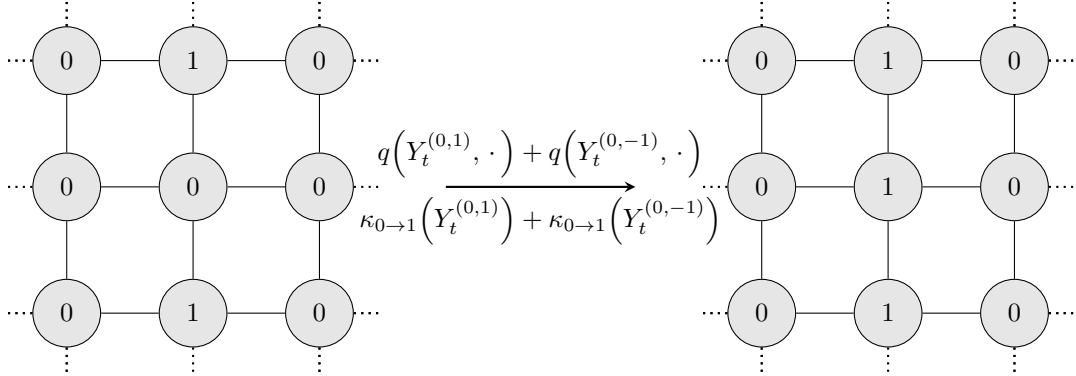
\begin{figure}[H]
  \centering
\begin{tikzpicture}[scale=0.84, every node/.style={font=\normalsize}, >=stealth, node distance=3cm and 2.5cm, auto,
  state/.style={circle, draw, fill=gray!20, minimum size=9mm, inner sep=1.5pt}
]

  \foreach \x in {0,1,2}{
    \foreach \y in {0,1,2}{
      \ifnum\x=1
        \ifnum\y=1
          \node[state] (L\x\y) at ({2*\x},{-2*\y}) {0};
        \else
          \ifnum\y=0
            \node[state] (L\x\y) at ({2*\x},{-2*\y}) {$1$};
          \else
            \node[state] (L\x\y) at ({2*\x},{-2*\y}) {$1$};
          \fi
        \fi
      \else
        \node[state] (L\x\y) at ({2*\x},{-2*\y}) {0};
      \fi
    }
  }

  \foreach \x in {0,1}{
    \foreach \y in {0,1,2}{
      \pgfmathtruncatemacro{\xp}{\x+1}
      \draw (L\x\y) -- (L\xp\y);
    }
  }
  \foreach \x in {0,1,2}{
    \foreach \y in {0,1}{
      \pgfmathtruncatemacro{\yp}{\y+1}
      \draw (L\x\y) -- (L\x\yp);
    }
  }

  \foreach \y in {0,1,2}{
    \draw[dotted, thick] (L2\y) -- ++(1,0);
  }
  \foreach \y in {0,1,2}{
    \draw[dotted, thick] (L0\y) -- ++(-1,0);
  }
  \foreach \x in {0,1,2}{
    \draw[dotted, thick] (L\x2) -- ++(0,-1);
  }
  \foreach \x in {0,1,2}{
    \draw[dotted, thick] (L\x0) -- ++(0,1);
  }
  \draw[dotted, thick] (L22) -- ++(1,0);
  \draw[dotted, thick] (L22) -- ++(0,-1);
  \draw[dotted, thick] (L02) -- ++(-1,0);
  \draw[dotted, thick] (L02) -- ++(0,-1);
  \draw[dotted, thick] (L20) -- ++(1,0);
  \draw[dotted, thick] (L20) -- ++(0,1);
  \draw[dotted, thick] (L00) -- ++(-1,0);
  \draw[dotted, thick] (L00) -- ++(0,1);

\draw[->, thick] (6,-2) -- node[below] {$\kappa_{1\to0}\Big(Y_t^{(0,1)}\Big)$} (9,-2);

  \foreach \x in {0,1,2}{
    \foreach \y in {0,1,2}{
      \ifnum\x=1
        \ifnum\y=1
          \node[state] (R\x\y) at ({11+2*\x},{-2*\y}) {$0$};
        \else
          \ifnum\y=0
            \node[state] (R\x\y) at ({11+2*\x},{-2*\y}) {$0$};
          \else
            \node[state] (R\x\y) at ({11+2*\x},{-2*\y}) {$1$};
          \fi
        \fi
      \else
        \node[state] (R\x\y) at ({11+2*\x},{-2*\y}) {0};
      \fi
    }
  }

  \foreach \x in {0,1}{
    \foreach \y in {0,1,2}{
      \pgfmathtruncatemacro{\xp}{\x+1}
      \draw (R\x\y) -- (R\xp\y);
    }
  }
  \foreach \x in {0,1,2}{
    \foreach \y in {0,1}{
      \pgfmathtruncatemacro{\yp}{\y+1}
      \draw (R\x\y) -- (R\x\yp);
    }
  }

  \foreach \y in {0,1,2}{
    \draw[dotted, thick] (R2\y) -- ++(1,0);
  }
  \foreach \y in {0,1,2}{
    \draw[dotted, thick] (R0\y) -- ++(-1,0);
  }
  \foreach \x in {0,1,2}{
    \draw[dotted, thick] (R\x2) -- ++(0,-1);
  }
  \foreach \x in {0,1,2}{
    \draw[dotted, thick] (R\x0) -- ++(0,1);
  }
  \draw[dotted, thick] (R22) -- ++(1,0);
  \draw[dotted, thick] (R22) -- ++(0,-1);
  \draw[dotted, thick] (R02) -- ++(-1,0);
  \draw[dotted, thick] (R02) -- ++(0,-1);
  \draw[dotted, thick] (R20) -- ++(1,0);
  \draw[dotted, thick] (R20) -- ++(0,1);
  \draw[dotted, thick] (R00) -- ++(-1,0);
  \draw[dotted, thick] (R00) -- ++(0,1);
\end{tikzpicture}

\caption{\label{fig:healing} Healing of a previously infected individual. The individual in the center is numbered~$(0,0)$.}
\end{figure}
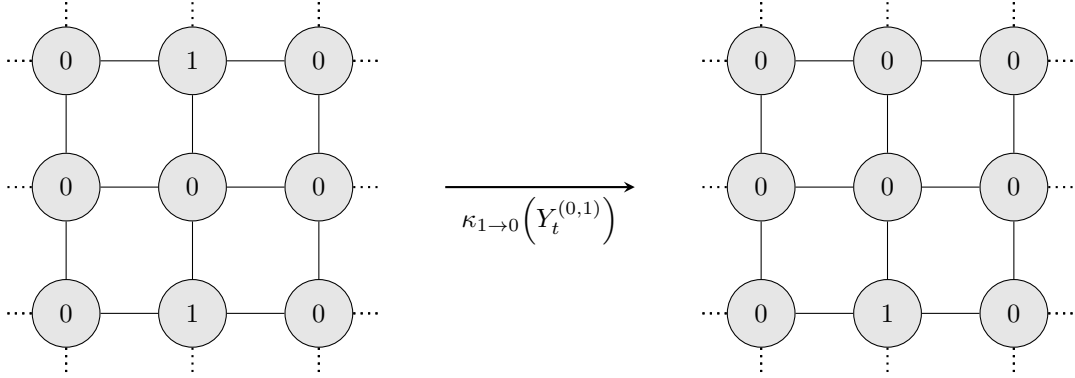

For each $n\geq 1$, we consider the contact process $\mathbf X^n$  with accelerated individual stochastic viral load $Y_{n\cdot}$ (instead of $Y$). We denote by $\mathbf I^n_t\in I$ the susceptible/infected configuration of the process at time $t$, which means that $(\mathbf I^n_t)_k=0$ (resp. $=1$) if the individual $k\in V_G$ is susceptible (resp. infected).

Before stating the convergence result, let us recall that  the classical contact process corresponds to the situation where there is no individual viral load, that is such that $D$ is a singleton, so that $\kappa_{0\to 1}$ and $\kappa_{1\to 0}$ are constants.
More precisely, a contact process on $I$ with infection rate $\lambda>0$ and healing rate $\mu>0$ is a pure jump process with  jumps rates given by
\[
i\to j\text{ at rate }\begin{cases}
\mu\quad&\textup{if}\quad i_k=1\,\textup{and}\,j=i-e_k,\\
\lambda\sum_{k\sim\ell}i_\ell\quad&\textup{if}\quad i_k=0\,\textup{and}\,j=i+e_k,\,\\
0\quad&\textup{otherwise}.
\end{cases}
\]
\begin{proposition}
\label{prop:contact}
Assume that there exists a probability measure  $\nu$ on $D$ such that
\[
\big\| \mathbf P_x(Y_t\in \cdot)-\nu\big\|_{TV}\underset{t\to\infty}{\longrightarrow} 0,\ \forall x\in D
\]
and that $\Big(\frac{1}{t}\int_0^t \kappa_{0\to 1}(Y_s)\,\mathrm ds\Big)_{t\geq 1}$ and $\Big(\frac{1}{t}\int_0^t \kappa_{1\to 0}(Y_s)\,\mathrm ds\Big)_{t\geq 1}$ are uniformly integrable for any initial position $Y_0\in D$. Then  $(\mathbf{I}_t^n)_{t\geq 0}$  converges in the Skorokhod topology on $D([0,+\infty),D)$ to a classical contact process on $I$ with infection rate $\nu(\kappa_{1\to 0})$ and healing rate $\nu(\kappa_{0\to 1})$.
\end{proposition}

Proposition~\ref{prop:contact} is an immediate consequence of our Theorem~\ref{thm:convSkorokhod}. To see this, let us write the contact process with individual stochastic viral load in the formalism of this abstract result.
We set $D_0=\{0\}$ and $D_1=D$,  $Y^1=Y$ and denote by $Y^0$ the constant process on $D_0$.
For all $i=(i_k)_{k\in V_G}\in I$, we set 
$E_i=\prod_{k\in V_G} D_{i_k}$ and assume that the dynamic of $X^\ori$ on $E_i$ is given by $(Y^{i_k,k})_{k\in V}$, where the $Y^{0,k}$ and $Y^{1,k}$ are independent copies of $Y^0$ and $Y^1$ respectively. One observes that the process on $E_i$ is positively Harris recurrent with stationary measure given by $\mu_i=\otimes_{k\in V_G} \nu_{i_k}$ where $\nu_0=\delta_{\{0\}}$ and $\nu_1=\nu$ (the argument is the same as in the branching process example).
Let $i\in I$ and $x_i\in E_i$, we have
\[
\mathbf P_{x_i}\big(\forall t\geq0, (Y_t^{i_1,1},Y_t^{i_2,2},\ldots)\in E_i\big)=1.
\]
Since we start with a finite number of infected individuals i.e. $\sum_{k\in V_G}i_k<\infty$ and using the lemma \ref{lem:tensorisation-cv-VT} we have,
\[
\frac{1}{t}\int_0^t\kappa_{1\to0}(Y_s^{i_1,1},Y_s^{i_2,2},\ldots,)\,\mathrm ds\underset{t\to\infty}{\longrightarrow}\mu_i(\kappa_{1\to0})\quad\mathbf P_{x_i}\textup{-p.s.}
\]
and
\[
\frac{1}{t}\int_0^t\kappa_{0\to1}(Y_s^{i_1,1},Y_s^{i_2,2},\ldots)\,\mathrm ds\underset{t\to\infty}{\longrightarrow}\mu_i(\kappa_{0\to1})\quad\mathbf P_{x_i}\textup{-p.s.}
\]

We have the following transition rates for the index process $\mathbf (\mathbf I^1_t)_{t\geq 0}$: for $x \in E_i$
\[
    i\to j\text{ at rate }
    \begin{cases}
    \kappa_{1\to0}(x_k)&\text{ if }x_k\in D_1\,(\text{i.e.}\,i_k=1)\text{ and }j=i-e_k\\
    \sum_{\ell\sim k,\,i_\ell=1}\kappa_{0\to1}( x_\ell)&\text{ if }x_k\in D_0\,(\text{i.e.}\,i_k=0)\text{ and }j=i+e_k\\
    0&\text{ otherwise.}
    \end{cases}
\]
where $e_k$ denotes the configuration with only zeros but a one at position $k$. We define $b : E \to \mathbb{R}_+$ by (note that the two sums are finite since  there is only a finite number of infective)
\[
b(x) = \sum_{k\in V_G} i_k\,\kappa_{1\to0}(x_k) + \sum_{k\in V_G}(1-i_k)\Big(\sum_{\ell\sim k}i_\ell\,\kappa_{0\to1}(x_\ell)\Big).
\]
For all $k\in V_G$ and $x\in E$ such that $x_k\in D_1$ let us define the probability kernel $K_{1\to0,k}:E\times E\longrightarrow[0,1]$
\[
K_{1\to0,k}( x, \deriv y) =\Big(\delta_{\{0\}}\otimes\prod_{\ell\in V_G\setminus\{k\}} \delta_{x_\ell}\Big)(\deriv y)
\]
and $K_{0\to1,k}:E\times E\longrightarrow[0,1]$ for $x\in E$ such that $x_k=0$ by
\[
K_{0\to1,k}(x,\mathrm dy)=\sum_{\ell\sim k}i_\ell\,\Big(q(x_\ell,\,\cdot\,)\otimes\prod_{m\in V_G\setminus\{k\}}\delta_{x_m}\Big)(\mathrm dy),
\]
where 
\[
q(x_\ell,\cdot)=\frac{\kappa_{0\to 1}(x_\ell)}{\sum_{j\sim k}i_j\,\kappa_{0 \to 1}(x_{j})}\,p(x_\ell,\,\cdot\,).
\]
Now, we can define the probability jumps transition $\pi:E\times E\to[0,1]$ by
\begin{align*}
\pi(x,\mathrm dy)&=\frac{1}{b(x)}\sum_{k\in V_G}i_k\,\kappa_{1\to0}(x_k)\,K_{1\to0,k}(x,\mathrm dy)\\
&+\frac{1}{b(x)}\sum_{k\in V_G}(1-i_k)\Big(\sum_{\ell\sim k}\,i_\ell\,\kappa_{0\to1}(x_\ell)\Big)\,K_{0\to1,k}(x,\mathrm dy),
\end{align*}

\noindent According to Theorem~\ref{thm:convSkorokhod}, the accelerated index process $\mathbf I^n$ converges in Skorokhod topology to the $I$-valued pure-jump Markov process $\mathbf I$ with jumps rate given by $1-e^{-t\mu_i(b)}$ with
\[
\mu_i(b)=\nu(\kappa_{1\to0})\sum_{k\in V_G}i_k+\nu(\kappa_{0\to1})\sum_{k\sim\ell}i_\ell
\]
and jumps transitions probabilities $(P(i,j))_{i,j\in I}$ 
\[
P(i,J)=\int_E\frac{\mu_i(\mathrm dx)}{\mu_i(b)}\,b(x)\,\nu(x,J),\quad\forall i\in I\,\textup{and}\,J\in\mathcal I.
\]
where for all $x$ $\nu(x,\,\cdot\,)$ is the push-forward measure of $\pi(x,\,\cdot\,)$ by $\phi$ where we remind $\phi(x)=i$ if and only if $x\in E_i$. Hence $P$ is given by
\begin{align*}
P(i,j)&=\begin{cases}
\int_E\frac{\mu_i(\mathrm dx)}{\mu_i(b)}\,\kappa_{1\to0}(x)\quad&\textup{if}\quad i_k=1\,\textup{and}\,j=i-e_k,\\
\int_E\frac{\mu_i(\mathrm dx)}{\mu_i(b)}\,\kappa_{0\to1}(x)\sum_{k\sim\ell}i_\ell\quad&\textup{if}\quad i_k=0\,\textup{and}\,j=i+e_k,\,\\
0\quad&\textup{otherwise}
\end{cases}\\
&=\begin{cases}
\frac{\nu(\kappa_{1\to0})}{\mu_i(b)}\quad&\textup{if}\quad i_k=1\,\textup{and}\,j=i-e_k,\\
\frac{\nu(\kappa_{0\to1})}{\mu_i(b)}\sum_{k\sim\ell}i_\ell\quad&\textup{if}\quad i_k=0\,\textup{and}\,j=i+e_k,\,\\
0\quad&\textup{otherwise}.
\end{cases}
\end{align*}
In particular, the transition rates of $\mathbf I$ are given by
\[
i\to j\text{ at rate }\begin{cases}
\nu(\kappa_{1\to0})\quad&\textup{if}\quad i_k=1\,\textup{and}\,j=i-e_k,\\
\nu(\kappa_{0\to1})\sum_{k\sim\ell}i_\ell\quad&\textup{if}\quad i_k=0\,\textup{and}\,j=i+e_k,\,\\
0\quad&\textup{otherwise},
\end{cases}
\]
which corresponds to a classical contact process with infection rate $\nu(\kappa_{0\to1})$ and healing rate $\nu(\kappa_{1\to0})$.

\section{Proofs of the convergence results at jump times}
\label{sec:proofs}
\subsection{Proof of Proposition~\ref{prop:convergence1}}
\label{sec:proofprop:convergence1}
Let us first remark that, if $\mu_i(b)=0$, then 
\[
\lim_{n \to \infty}G\Big(\int_0^{t} \,b(X_{ns}^{\ori}(v^\ori))\,\deriv s\Big) =\lim_{n \to \infty} G\Big(\frac{1}{n} \int_0^{n t} b(X_{s}^\ori) \, \deriv s\Big) = G(t \mu_i(b)) = G(0)=0.
\]
In particular,
\[
\tau_1^n = \inf\{t\geq 0:  G\Big(\int_0^t \,b(X_{ns}^{\ori}(v^\ori))\,\deriv s\Big)> U\}\overset{\textup{a.s.}}{\underset{n\to\infty}{\longrightarrow}} +\infty.
\]
This shows that
\[
\lim_{n\to\infty}\espn_{x_i}\big[f(\mathbf{X}_{\tau_1^n}^n)\,\mathbf{1}_{\{\tau_1^n\leq T\}}\big]=0,
\]
which gives the result.

From now on, we assume that $\mu_i(b)>0$. 

\textit{Step 1.} For this proof we first suppose that $G':[0,\infty)\longrightarrow [0,\infty)$ is both smooth and compactly supported.
Let $T>0$ a real positive number, $f:E\longrightarrow\mathbb{R}$ a bounded $\mathcal E$-measurable function. Since $f$ is bounded, even if it means replacing $f$ with $f_+-f_-$, we will assume that $f$ is non-negative. Now, $f:E\longrightarrow\mathbb R_+$ is a bounded, positive and $\mathcal E$-measurable function. Then, using Proposition~\ref{prop:moyenne1}, for all $i\in I$, $x_i\in E_i$ and $n\geq1$ we have
\[
\mathbf E_{x_i}^n\big[f(\mathbf{X}_{\tau_1^n}^n)\,\mathbf{1}_{\{\tau_1^n\leq T\}}\big]=\mathbf{E}_{x_i}^\ori\Big[\int_0^T\deriv t\,G'\Big(\int_0^tb(X_{ns}^\ori)\,\deriv s\Big)\,b(X_{nt}^\ori)\,\pi(X_{nt}^\ori,f)\Big]
\]
which is equal, by adding and subtracting to
\begin{gather}
    \mathbf{E}_{x_i}^\ori\Big[\int_0^T\deriv t\,\Big(G'\Big(\int_0^tb(X_{ns}^\ori)\,\deriv s\Big)-G'(t\mu_i(b)))\,b(X_{nt}^\ori)\,\pi(X_{nt}^\ori,f)\Big]\notag\\
    +\mathbf{E}_{x_i}^\ori\Big[\int_0^T\deriv t\,G'(t\mu_i(b))\,b(X_{nt}^\ori)\,\pi(X_{nt}^\ori,f)\Big]\label{eq:diffG}.
\end{gather}
Then, for all $n\geq1$ and $t\geq0$, by Proposition~4.10 of Chapter~0 in \cite{RevuzYor}, the fisrt term above is equal to
\[
    \mathbf E_{x_i}^\ori\Big[\int_0^{I_n(T)}\deriv t\,\big(G'(t)-G'(\mu_i(b)\,I_n^{-1}(t))\big)\,\pi(X_{nI_n^{-1}(t)}^\ori,f)\Big].
\]
By adding and subtracting, the  term above can be rewritten as follow
\begin{gather}
    \mathbf E_{x_i}^\ori\Big[\int_0^{T\mu_i(b)}\deriv t\,\big(G'(t)-G'(\mu_i(b)\,I_n^{-1}(t))\big)\,\pi(X_{nI_n^{-1}(t)}^\ori,f)\Big]\notag\\
    +\mathbf E_{x_i}^\ori\Big[\int_{T\mu_i(b)}^{I_n(T)}\deriv t\,\big(G'(t)-G'(\mu_i(b)\,I_n^{-1}(t))\big)\,\pi(X_{nI_n^{-1}(t)}^\ori,f)\Big]
    \label{eq:proofconvmean1}
\end{gather}
The first term in \eqref{eq:proofconvmean1} can be bounded above by
\[
    \|f\|_\infty\int_0^{T\mu_i(T)}\mathbf E_{x_i}^\ori\big[G'(t)-G'\big(\mu_i(b)I_n^{-1}(t)\big)\big]\,\deriv t.
\]
By the ergodicity assumption~\eqref{hyp:ergo}, we have, almost surely, $\lim_{n\to\infty} I_n(t)=t\mu_i(b)$.
Since $I_\infty:t\mapsto t\mu_i(b)$ is increasing with continuous inverse, 
we deduce that, almost surely, $I_n^{-1}(t)$ converges to $I^{-1}_\infty(t)=t/\mu_i(b)$ for all $t\geq 0$. Since $G'$ is bounded and continuous, we have
\[
    \mathbf E_{x_i}^\ori\big[G'(t)-G'\big(\mu_i(b)\,I_n^{-1}(t)\big)\big]\underset{n\to\infty}{\longrightarrow}0,
\]
and by the dominated convergence theorem
\begin{equation}
    \int_0^{T\mu_i(b)}\mathbf E_{x_i}^\ori\big[G'(t)-G'\big(\mu_i(b)\,I_n^{-1}(t)\big)\big]\,\deriv t\underset{n\to\infty}{\longrightarrow}0.
    \label{lim:proofconvmean1}
\end{equation}
The second term in \eqref{eq:proofconvmean1} is dominated by
\begin{gather*}
    \mathbf E_{x_i}^\ori\Big[\int_{T\mu_i(b)}^{I_n(T)}\deriv t\,\big(G'(t)-G'(\mu_i(b)\,I_n^{-1}(t))\big)\,\pi(X_{nI_n^{-1}(t)}^\ori,f)\Big]\\
    \leq 2\,\|G'\|_\infty\,\|f\|_\infty\,\mathbf E_{x_i}^\ori\big[\lvert I_n(T)-T\mu_i(b)\rvert\big]
\end{gather*}
which goes to $0$ when $n$ goes to infinity, since the family of random variables $(I_n(T))_{n\geq1}$ is uniformly integrable and $I_n(T)$ converges in probability to $T\mu_i(b)$. Then
\begin{equation}
    \mathbf E_{x_i}^\ori\Big[\int_{T\mu_i(b)}^{I_n(T)}\deriv t\,\big(G'(t)-G'(\mu_i(b)I_n^{-1}(t))\big)\,\pi(X_{nI_n^{-1}(t)}^\ori,f)\Big]\underset{n\to\infty}{\longrightarrow}0.
    \label{lim:proofconvmean2}
\end{equation}
So, by \eqref{eq:proofconvmean1}, \eqref{lim:proofconvmean1} and \eqref{lim:proofconvmean2}  we have shown that
\begin{equation}
\mathbf{E}_{x_i}^\ori\Big[\int_0^T\deriv t\,\Big(G'\Big(\int_0^tb(X_{ns}^\ori)\,\deriv s\Big)-G'(t\mu_i(b))\Big)\,b(X_{nt}^\ori)\,\pi(X_{nt}^\ori,f)\Big]\underset{n\to\infty}{\longrightarrow}0.\label{lim:proofconvmean3}
\end{equation}
Doing an integration by parts, for the second term  in~\eqref{eq:diffG}, we have
\begin{align}
\mathbf E_{x_i}^\ori\Big[\int_0^T\deriv t\,G'(t\mu_i(b))\,b(X_{nt}^\ori)\,\pi(X_{nt}^\ori,f)\Big]&= \mathbf E_{x_i}^\ori\Big[G'(T\mu_i(b))\,\frac{1}{n}\int_0^{nT}\deriv t\,b(X_{t}^\ori)\,\pi(X_{t}^\ori,f)\Big]
\label{eq:proofconvmean3}\\
&-\mathbf E_{x_i}^\ori\Big[\int_0^T\deriv t\,\mu_i(b)\, G''(t\mu_i(b))\,\frac{1}{n}\int_0^{nt}\deriv s\,b(X_{s}^\ori)\,\pi(X_{s}^\ori,f)\Big].\notag
\end{align}
Since $b\in L^1(\mu_{i})$ and $f$ is bounded, $x\mapsto b(x)\,\pi(x,f)\in L^1(\mu_{i})$, and, in addition, the sequence of random variables $\frac{1}{n}\int_0^{nT}\deriv t\,b(X_t^\ori)\,\pi(X_t^\ori,f)$, $n\geq 1$, is uniformly integrable (by assumption), we deduce that
\begin{equation}
    \mathbf E_{x_i}^\ori\Big[\frac{1}{n}\int_0^{nT}\deriv t\,b(X_t^\ori)\,\pi(X_t^\ori,f)\Big]\underset{n\to\infty}{\longrightarrow}T\int_E\mu_{i}(\deriv x)\,b(x)\,\pi(x,f).
    \label{lim:proofconvmean4}
\end{equation}
For the second term in \eqref{eq:proofconvmean3} we have
\begin{multline}
    \mathbf E_{x_i}^\ori\Big[\int_0^T\deriv t\,\mu_i(b)\,G''(t\mu_i(b))\int_0^tb(X_{ns}^\ori)\,\pi(X_{ns}^\ori,f)\,\deriv s\Big]\\
    =\mu_i(b)\int_0^T\deriv t\,G''(t\mu_i(b))\mathbf E_{x_i}^\ori\Big[\int_0^tb(X_{ns}^\ori)\,\pi(X_{ns}^\ori,f)\,\deriv s\Big]
    \label{eq:proofconvmean4}
\end{multline}
As $b$ and $f$ are positive function, we remark that for all $n\geq1$ the function 
\[
    J_n:\;t\in[0,T]\longmapsto\mathbf E_{x_i}^\ori\Big[\int_0^tb(X_{ns}^\ori)\,\pi(X_{ns}^\ori,f)\,\deriv s\Big]
\]
is non-decreasing. Since $f$ is bounded,  by Assumption~\ref{hyp:b-UI}, the family $\big(\frac{1}{n}\int_0^{nt}b(X_s^\ori)\,\pi(X_s^\ori,f)\,\mathrm ds\big)_{n\geq1}$ is uniformly integrable for all $t\geq0$ under $\mathbf P_x^\ori$ for all $x\in E$. In addition,  by Lemma~\ref{lemma:assumptionErgo}, $b(\cdot)\pi(\,\cdot\,,f)\in \mathcal A_i$. Hence, by Assumption~\ref{assumption:ergo} and uniform integrability, the limit of $J_n$ is 
\[
    J_\infty:\;t\in[0,T]\longmapsto t\int_E\mu_i(\deriv x)\,b(x)\,\pi(x,f)
\]
which is continuous. So, by the Dini's theorem
\[
    \sup_{t\in[0,T]}\big\lvert J_n(t)-J_\infty(t)\big\rvert\underset{n\to\infty}{\longrightarrow}0,
\]
and, since $G''$ is bounded we also have
\[
    \sup_{t\in[0,T]}\big\lvert G''\big(t\mu_i(b)\big)\,\big(J_n(t)-J_\infty(t)\big)\big\rvert\underset{n\to\infty}{\longrightarrow}0.
\]
Finally, by combining this limit and \eqref{eq:proofconvmean4}, we get
\begin{multline}
    \mathbf E_{x_i}^\ori\Big[\int_0^T\deriv t\,\mu_i(b)\, G''(t\mu_i(b))\,\frac{1}{n}\int_0^{nt}\deriv s\,b(X_{s}^\ori)\,\pi(X_{s}^\ori,f)\Big]\\
    \underset{n\to\infty}{\longrightarrow}\int_0^T\deriv t\,\mu_i(b)\,G''(t\mu_i(b))\,t\int_E\mu_i(\deriv x)\,b(x)\,\pi(x,f).
    \label{lim:proofconvmean5}
\end{multline}
Finally, by \eqref{eq:proofconvmean3}, \eqref{lim:proofconvmean4} and \eqref{lim:proofconvmean5} we have
\begin{multline*}
    \mathbf E_{x_i}^\ori\Big[\int_0^T\deriv t\,G'(t\mu_i(b))\,b(X_{nt}^\ori)\,\pi(X_{nt}^\ori,f)\Big]\\
    \underset{n\to\infty}{\longrightarrow} G'(T\mu_i(b))\,T\int_E\mu_i(\deriv x)\,b(x)\,\pi(x,f)-\int_0^T\,\deriv t\,  G''(t\mu_i(b))\,t\int_E\mu_i(\deriv x)\,b(x)\,\pi(x,f).
\end{multline*}
Again by doing an integration by parts, the term above equals
\begin{gather*}
G'(T\mu_i(b))\,T\int_E\mu_i(\deriv x)\,b(x)\,\pi(x,f)- G'(T\mu_i(b))\,T\int_E\mu_i(\deriv x)\,b(x)\,\pi(x,f)\\
+\int_0^T\deriv t\, G'(t\mu_i(b))\int_E\mu_i(\deriv x)\,b(x)\,\pi(x,f),
\end{gather*}
so that,
\[
    \mathbf E_{x_i}^\ori\Big[\int_0^T\deriv t\, G'(t\mu_i(b))\,b(X_{nt}^\ori)\int_E\pi(X_{nt}^\ori,\deriv y)\,f(y)\Big]\underset{n\to\infty}{\longrightarrow} \int_0^T\deriv t\, G'(t\mu_i(b))\int_E\mu_i(\deriv x)\,b(x)\,\pi(x,f).
\]
Then, using this last converge, \eqref{eq:diffG} and \eqref{lim:proofconvmean1}, we deduce that
\[
\lim_{n\to\infty}\esp_{x_i}\big[f(\mathbf{X}_{\tau_1^n}^n)\,\mathbf{1}_{\{\tau_1^n\leq T\}}\big]=\int_0^T G'(t\mu_i(b))\,\deriv t\int_E\mu_i(\deriv x)\,b(x)\,\pi(x,f).
\label{eq:smoothlimit}
\]

\textit{Step 2.} If $G'\in L^1(\mathbb R_+)$ is not smooth and compact supported, by density, let $\varepsilon>0$ and $G'_\varepsilon\in C_c^\infty(\mathbb{R}_+)$ such that 
\[
    \|G'-G'_\varepsilon\|_{L^1}\leq\varepsilon\,/(3\,\|f\|_\infty).
\]
Then, we rewrite the quantity below
\[
\mathbf E_{x_i}^n\big[f(\mathbf{X}_{\tau_1^n}^n)\,\mathbf{1}_{\{\tau_1^n\leq T\}}\big]=\mathbf{E}_{x_i}^\ori\Big[\int_0^T\deriv t\,G'\Big(\int_0^tb(X_{ns}^\ori)\,\deriv s\Big)\,b(X_{nt}^\ori)\,\pi(X_{nt}^\ori,f)\Big]
\]
as follows
\begin{gather}
\mathbf{E}_{x_i}^\ori\Big[\int_0^T\deriv t\,\Big(G'\Big(\int_0^tb(X_{ns}^\ori)\,\deriv s\Big)-G'_\varepsilon\Big(\int_0^tb(X_{ns}^\ori)\,\deriv s\Big)\Big)\,b(X_{nt}^\ori)\,\pi(X_{nt}^\ori,f)\Big]\notag\\
+\mathbf{E}_{x_i}^\ori\Big[\int_0^T\deriv t\,G'_\varepsilon\Big(\int_0^tb(X_{ns}^\ori)\,\deriv s\Big)\,b(X_{nt}^\ori)\,\pi(X_{nt}^\ori,f)\Big].\label{eq:finalterm}
\end{gather}
The first term in \eqref{eq:finalterm} is bounded above by
\[
    \|f\|_\infty\,\mathbf{E}_{x_i}^\ori\Big[\int_0^T\deriv t\,\Big(G'\Big(\int_0^tb(X_{ns}^\ori)\,\deriv s\Big)-G'_\varepsilon\Big(\int_0^tb(X_{ns}^\ori)\,\deriv s\Big)\Big)\,b(X_{nt}^\ori)\Big]
\]
which is equal to
\begin{equation}
     \|f\|_\infty\,\mathbf{E}_{x_i}^\ori\Big[\int_0^{I_n(T)}(G'(t)-G'_\varepsilon(t))\,\deriv t\Big]\leq\|f\|_\infty\,\|G'-\tilde G'\|_{L^1}\leq\varepsilon/3.
     \label{eq:lastinegality}
\end{equation}
By the step 1, the second term in \eqref{eq:finalterm} goes to 
\begin{equation}
    \int_0^T G'_\varepsilon(t\mu_i(b))\,\deriv t\int_E\mu_i(\deriv x)\,b(x)\,\pi(x,f)
    \label{lim:finalterm}
\end{equation}
Then using \eqref{eq:smoothlimit}, \eqref{eq:lastinegality} and \eqref{lim:finalterm} we deduce for all $\varepsilon>0$, there exists $n_\varepsilon\geq0$ such that for all $n\geq n_\varepsilon$
\begin{align*}
    &\Big\lvert\mathbf E_{x_i}^n\big[f(\mathbf{X}_{\tau_1^n}^n)\,\mathbf{1}_{\{\tau_1^n\leq T\}}\big]-\int_0^T G'(t\mu_i(b))\,\deriv t\int_E\mu_i(\deriv x)\,b(x)\,\pi(x,f)\Big\rvert\\
    &\leq\Big\lvert\mathbf{E}_{x_i}^\ori\Big[\int_0^T\deriv t\,\Big(G'\Big(\int_0^tb(X_{ns}^\ori)\,\deriv s\Big)-G'_\varepsilon\Big(\int_0^tb(X_{ns}^\ori)\,\deriv s\Big)\Big)\,b(X_{nt}^\ori)\,\pi(X_{nt}^\ori,f)\Big]\Big\rvert\\
    &+\Big\lvert\mathbf{E}_{x_i}^\ori\Big[\int_0^T\deriv t\,G'_\varepsilon\Big(\int_0^tb(X_{ns}^\ori)\,\deriv s\Big)\,b(X_{nt}^\ori)\,\pi(X_{nt}^\ori,f)\Big]-\int_0^T G'_\varepsilon(t\mu_i(b))\,\deriv t\int_E\mu_i(\deriv x)\,b(x)\,\pi(x,f)\Big\rvert\\
    &+\Big\lvert\int_0^T \big(G'_\varepsilon(t\mu_i(b))- G'(t\mu_i(b))\big)\,\deriv t\int_E\mu_i(\deriv x)\,b(x)\,\pi(x,f)\Big\rvert\leq\varepsilon,
\end{align*}
which concluded the proof of the proposition.

\subsection{Proof of Theorem~\ref{thm:cv-X-jump_times}}
\label{sec:proofthm:cv-I-jump_times}
In order to prove Theorem~\ref{thm:cv-X-jump_times}, we need the following lemma:
\begin{lemma}
\label{lem:ergon}
Let $x_0\in E$ and $(f_n)_{n\geq0}$ a sequence of functions such that for all $n\geq0$ $f_n\in L^1(\mu_{x_0})$. We assume that there exists $f\in L^1(\mu_{x_0})$ such that for all $x\in E$, $\lim_n \Delta_n(x):=f_n(x)-f(x)=0$ and 
\[
    \bar \Delta:=\sup_n |\Delta_n|\in L^1(\mu_{x_0}).
\]
Then, $\mathbf{P}_{x_0}^\ori$-almost surely
\[
    \lim_{n\to\infty}\frac{1}{t}\int_0^tf_n(X_{ns}^\ori)\,\deriv s=\mu_{x_0}(f).
\]
\end{lemma}
\begin{proof}
Let $\varepsilon>0$, we set for all $x\in E$
\[
    n_\varepsilon(x)=\inf\{n\geq0:\Delta_m(x)\leq\varepsilon/2, \,\forall m\geq n\}
\]
and for all $n\geq0$
\[
    E_\varepsilon^n=\{x\in E:n_\varepsilon(x)\leq n\}.
\]
By assumption, for all $x\in E$, $\Delta_n(x)$ goes to zero when $n$ goes to infinity, then, for all $x\in E$, $n_\varepsilon(x)<\infty$ and $(E_\varepsilon^n)_{n\geq0}$ is a non-decreasing sequence so $\cup_{n\geq0}E_\varepsilon^n=E$. Then, there exists $n_\varepsilon\geq0$ such that  {$\mu_{x_0}( \bar\Delta \mathbf 1_{(E_\varepsilon^{n_\varepsilon})^c})\leq \varepsilon/2$}. Let $n\geq n_\varepsilon$, 
\begin{align*}
\frac{1}{t}\int_0^t\Delta_n(X_{ns}^\ori)\,\deriv s&=\frac{1}{t}\int_0^t\Delta_n(X_{ns}^\ori)\,\mathbf{1}_{\{X_{ns}^\ori\in E_\varepsilon^{n_\varepsilon}\}}\,\deriv s+\frac{1}{t}\int_0^t\Delta_n(X_{ns}^\ori)\,\mathbf{1}_{\{X_{ns}^\ori\not\in E_\varepsilon^{n_\varepsilon}\}}\,\deriv s\\
&=\frac{1}{nt}\int_0^{nt}\Delta_n(X_s^\ori)\,\mathbf{1}_{\{X_{s}^\ori\in E_\varepsilon^{n_\varepsilon}\}}\,\deriv s+\frac{1}{nt}\int_0^{nt}\Delta_n(X_s^\ori)\,\mathbf{1}_{\{X_{s}^\ori\not\in E_\varepsilon^{n_\varepsilon}\}}\,\deriv s\\
&\leq\frac{\varepsilon}{2nt}\int_0^{nt}\mathbf{1}_{\{X_{s}^\ori\in E_\varepsilon^{n_\varepsilon}\}}\,\deriv s+\frac{1}{nt}\int_0^{nt}\bar\Delta(X_s^\ori) \mathbf{1}_{\{X_{s}^\ori\not\in E_\varepsilon^{n_\varepsilon}\}}\,\deriv s,
\end{align*}
which converges $\mathbf{P}_{x_0}^\ori$-p.s. to $\varepsilon\,\mu_{x_0}(E_\varepsilon^{n_\varepsilon})/2+\mu_{x_0}( \bar\Delta \mathbf 1_{(E_\varepsilon^{n_\varepsilon})^c})\leq \varepsilon$ when $n$ goes to infinity. Then,
\[
\limsup_n\frac{1}{t}\int_0^t\Delta_n(X_{ns}^\ori)\,\deriv s\leq\varepsilon,
\]
i.e.
\[
\lim_{n\to\infty}\frac{1}{t}\int_0^t\Delta_n(X_{ns}^\ori)\,\deriv s=\lim_{n\to\infty}\frac{1}{t}\int_0^t\big(f_n(X_{ns}^\ori)-f(X_{ns}^\ori)\big)\,\deriv s=0.
\]
Then,
\[
\frac{1}{t}\int_0^tf_n(X_{ns}^\ori)\,\deriv s-\mu_{x_0}(f)=\frac{1}{t}\int_0^t\big(f_n(X_{ns}^\ori)-f(X_{ns}^\ori)\big)\,\deriv s+\frac{1}{t}\int_0^tf(X_{ns}^\ori)\,\deriv s-\mu_{x_0}(f)
\]
and each term goes to zero when $n$ goes to infinity which concludes the proof.
\end{proof}
We can now move on to the proof of Theorem~\ref{thm:cv-X-jump_times}.
\begin{proof}[Proof of Theorem~\ref{thm:cv-X-jump_times}]
We make the proof by induction on $N\geq1$. Let $x\in E$, $T_1,\ldots,T_N>0$ positive real numbers and $f:E^{N}\longrightarrow\mathbb{R}$ a $\mathcal E^{\otimes N}$-measurable bounded functions.  Even if it means replacing $f$ with the difference between its positive and negative parts we will assume that $f$ is a positive function. For $N=1$, the result is given by the proposition \ref{prop:convergence1}
\begin{align*}
\lim_{n\to\infty}\espn_{x}\big[f(\mathbf{X}_{\tau_1^n}^n)\,\mathbf{1}_{\{\tau_1^n\leq T\}}\big]&=\Big(\int_0^T\mu_{\phi(x)}(b)\,G'(t\mu_{\phi(x)}(b))\,\deriv t\Big)\times\Big(\int_E\frac{\mu_{\phi(x)}(\deriv y)}{\mu_{\phi(x)}(b)}\,b(y)\,\pi(y,f)\Big)\\
&=\esp_x\big[f(\mathbf Y_1)\,\mathbf 1_{\{\tau_1\leq T\}}\big].
\end{align*}
Now, for $N+1$, by the Markov property at time $\tau_1^n$ we have
\[
\espn_{x}\big[f(\mathbf{X}_{\tau_1^n}^n,\ldots,\mathbf{X}_{\tau_{N+1}^n}^n)\,\mathbf{1}_{\{\tau_1^n\leq T_1\}}\ldots\mathbf{1}_{\{\tau_{N+1}^n-\tau_{N}^n\leq T_{N+1}\}}\big]=\espn_x\big[\mathbf 1_{\{\tau_1^n\leq T_1\}}\,\psi^n(\mathbf X_{\tau_1^n}^n)\big]
\]
where 
\[
\psi^n(y):=\espn_y\big[f(y,\mathbf X_{\tau_1^n}^n,\ldots,\mathbf X_{\tau_N^n}^n)\,\mathbf 1_{\{\tau_1^n\leq T_2\}}\ldots\mathbf 1_{\{\tau_N^n-\tau_{N-1}^n\leq T_{N+1}\}}\big]
\]
for all $y\in E$ and $n\geq1$. Then, by the induction property applied to the function 
\[
(x_1,\ldots,x_N)\in E^N\mapsto f(y,x_1,\ldots,x_N)
\]
and to the times $T_2, \ldots, T_{N+1}$, we get, for all $y\in E$,
\[
\psi^n(y)\underset{n\to\infty}{\longrightarrow}\psi(y):=\mathbf E_y\big[f(y,\mathbf Y_1,\ldots,\mathbf Y_{N+1})\,\mathbf 1_{\{\tau_1\leq T_2\}}\ldots\mathbf 1_{\{\tau_{N+1}-\tau_N\leq T_{N+1}\}}\big].
\]
Using the definition of the process and the law of $\tau_1^n$ we have
\begin{align*}
\espn_x\big[\mathbf 1_{\{\tau_1^n\leq T_1\}}\,\psi^n(\mathbf X_{\tau_1^n}^n)\big]&=\mathbf E_x^\ori\Big[\int_0^{T_1}\deriv t_1\,G'\big(I_{t_1}^n(b)\big)\,b(X_{nt_1}^\ori)\int_E\pi(X_{nt_1}^\ori,\deriv y_1)\,\psi^n(y_1)\Big]\\
&=\mathbf E_x^\ori\Big[\int_0^{T_1}\deriv t_1\,G'\big(I_{t_1}^n(b)\big)\,\varphi^n(X_{nt_1}^\ori)\Big],
\end{align*}
where we set
\[
\varphi^n(y):=b(y)\int_E\pi(y,\deriv y_1)\,\psi^n(y_1),\quad\forall y\in E.
\]
\textit{Step 1.}
From now we will assume that $G$ is both smooth and compactly supported i.e. $G\in C_c^\infty(\mathbb R)$. By adding and subtracting, the previous term equals
\[
\mathbf E_x^\ori\Big[\int_0^{T_1}\deriv t_1\,\big(G'(I_{t_1}^n(b))-G'(t_1\mu_x(b))\big)\varphi^n(X_{nt_1}^\ori)\Big]+\mathbf E_x^\ori\Big[\int_0^{T_1}\deriv t_1\,G'(t_1\mu_x(b))\,\varphi^n(X_{nt_1}^\ori)\Big].
\]
Since 
\[
\lvert\psi^n(y)\rvert\leq\|f\|_\infty,\quad\forall n\geq1, y\in E,
\]
the first term above can be dominated by
\[
\|f\|_\infty\,\mathbf E_x^\ori\Big[\int_0^{T_1}\big\lvert G'(I_{t_1}^n(b))-G'(t_1\mu_x(b))\big\rvert\,b(X_{nt_1}^\ori)\,\deriv t_1\Big],
\]
which goes to zero when $n$ goes to zero using the same argument as in the proof of the Proposition \ref{prop:convergence1}. By integrating by parts, we have
\begin{align*}
\mathbf E_x^\ori\Big[\int_0^{T_1}\deriv t_1\,G'(t_1\mu_x(b))\,\varphi^n(X_{nt_1}^\ori)\Big]&=\mathbf E_x^\ori\Big[G'(T_1\mu_x(b))\int_0^{T_1}\varphi^n(X_{ns}^\ori)\,\deriv s\Big]\\
&-\mathbf E_x^\ori\Big[\int_0^{T_1}\deriv t_1\,\mu_x(b)\,G''(t_1\mu_x(b))\int_0^{t_1}\varphi^n(X_{ns}^\ori)\,\deriv s\Big].
\end{align*}
Using the induction hypothesis and the dominated convergence theorem, we have
\[
\varphi^n(y)\underset{n\to\infty}{\longrightarrow}\varphi(y):=b(y)\int_E\pi(y,\deriv y_1)\,\psi(y_1)
\]
for all $y\in E$ and 
\[
\varphi^n\leq \|f\|_\infty\,b\in L^1(\mu).
\]
So, by the previous Lemma \ref{lem:ergon}, we have, for all $t>0$
\begin{align*}
\mathbf E_x^\ori\Big[\int_0^t\varphi^n(X_{ns}^\ori)\,\deriv s\Big]\underset{n\to\infty}{\longrightarrow}t\mu_x(\varphi).
\end{align*}
Moreover, since $f$ is positive, $\varphi^n$ is also positive for all $n\geq1$, so the function 
\[
t\mapsto\mathbf E_x^\ori\Big[\int_0^t\varphi^n(X_{ns}^\ori)\,\deriv s\Big]
\]
is non-decreasing. Then, by the Dini's Theorem the previous convergence is uniform. Reminding that $G\in C_c^\infty(\mathbb R)$ we finally 
\begin{align*}
\mathbf E_x^\ori\Big[\int_0^{T_1}\deriv t_1\,G'(t_1\mu_x(b))\,\varphi^n(X_{nt_1}^\ori)\Big]&\underset{n\to\infty}{\longrightarrow}G'(T_1\mu_x(b))\,T_1\mu_x(\varphi)\\
&-\int_0^{T_1}\mu_x(b)\,G''(t_1\mu_x(b))\,t_1\mu_x(\varphi)\,\deriv t_1.
\end{align*}
Again by integrating by parts, the previous term is equal to
\[
\int_0^{T_1}\deriv t_1G'(t_1\mu_x(b))\,\mu_x(\varphi),
\]
which is equal to
\begin{multline*}
\int_0^{T_1}\deriv t_1\,\mu_{\phi(x)}(b)\,G'(t_1\mu_{\phi(x)}(b))\,\int_E\frac{\mu_{\phi(x)}(\mathrm dy)}{\mu_x(b)}\,b(y)\int_E\pi(y,\mathrm dy_1)\\
\times\mathbf E_{y_1}\big[f(y_1,\mathbf Y_1,\ldots,\mathbf Y_{N+1})\,\mathbf 1_{\{\tau_1\leq T_2\}}\ldots\mathbf 1_{\{\tau_{N+1}\leq T_{N+1}\}}\big].
\end{multline*}
By the strong Markov property of $\mathbf Y$ at time $\tau_1$ the previous term is equal to 
\[
\mathbf E_x\big[f(\mathbf Y_1,\ldots,\mathbf Y_{N+1})\,\mathbf 1_{\{\tau_1\leq T_1\}}\ldots\mathbf 1_{\{\tau_{N+1}\leq T_{N+1}\}}\big]
\]
which is exactly the induction hypothesis at rank $N+1$.

\textit{Step 2.} Now if $G'\in L^1(\mathbb R_+)$ ($G'$ is a probability density on $\mathbb R_+$) is not both smooth and compacted supported, by density of $C_c^\infty(\mathbb R_+)$ in $L^1(\mathbb R_+)$, there exists $G'_\varepsilon$ in $C_c^\infty(\mathbb R_+)$ such that
\[
    \|G'-G'_\varepsilon\|_{L^1}\leq\varepsilon/(3\,\|f\|_\infty).
\]
We recall that
\[
\espn_{x}\big[f(\mathbf{X}_{\tau_1^n}^n,\ldots,\mathbf{X}_{\tau_{N+1}^n}^n)\,\mathbf{1}_{\{\tau_1^n\leq T_1\}}\ldots\mathbf{1}_{\{\tau_{N+1}^n-\tau_{N}^n\leq T_{N+1}\}}\big]
=\mathbf E_x^\ori\Big[\int_0^{T_1}\deriv t_1\,G'(I_{t_1}^n(b))\,\varphi^n(X_{nt_1}^\ori)\Big]
\]
which is equal to
\[
\mathbf E_x^\ori\Big[\int_0^{T_1}\deriv t_1\,\big(G'(I_{t_1}^n(b))-G_\varepsilon'(I_{t_1}^n(b)))\,\varphi^n(X_{nt_1}^\ori)\Big]+\mathbf E_x^\ori\Big[\int_0^{T_1}\deriv t_1\,G_\varepsilon'\big(I_{t_1}^n(b)\big)\,\varphi^n(X_{nt_1}^\ori)\Big].
\]
By substitution, the first term above is equal to
\[
\mathbf E_x^\ori\Big[\int_0^{I_{T_1}^n(b)}\deriv t_1\,\big(G'(t_1)-G_\varepsilon'(t_1)\big)\,\pi(X_{nI_{t_1}^{-n}(b)}^\ori,\psi^n)\Big],
\]
which is bounded above by
\[
\|f\|_\infty\,\mathbf E_x^\ori\Big[\int_0^{I_{T_1}^n(b)}\big\lvert G'(t_1)-G_\varepsilon'(t_1)\big\rvert\,\deriv t_1\Big]\leq\|f\|_\infty\,\|G'-G_\varepsilon'\|_{L^1}\underset{n\to\infty}{\longrightarrow}0.
\]
For the second term, by addition and subtraction, we have
\begin{align*}
\mathbf E_x^\ori\Big[\int_0^{T_1}\deriv t_1\,G_\varepsilon'(I_{t_1}^n(b))\,\varphi^n(X_{nt_1}^\ori)\Big]&=\mathbf E_x^\ori\Big[\int_0^{T_1}\deriv t_1\,\big(G_\varepsilon'(I_{t_1}^n(b))-G_\varepsilon'(t_1\mu_x(b))\,\varphi^n(X_{nt_1}^\ori)\Big]\\
&+\mathbf E_x^\ori\Big[\int_0^{T_1}\deriv t_1\,G_\varepsilon'(t_1\mu_x(b))\,\varphi^n(X_{nt_1}^\ori)\Big].
\end{align*}
As in the first step, the first term above goes to zero when $n$ goes to infinity and the second one goes to
\[
\mu_x(\varphi)\int_0^{T_1}G_\varepsilon'\big(t_1\mu_x(b)\big)\,\deriv t_1
\]
when $n$ goes to infinity. Again, by the strong Markov property of $\mathbf Y$ at time $\tau_1$,
\[
    \espn_{x}\big[f(\mathbf{Y}_1,\ldots,\mathbf{Y}_{N+1})\,\mathbf{1}_{\{\tau_1\leq T_1\}}\ldots\mathbf{1}_{\{\tau_{N+1}-\tau_{N}\leq T_{N+1}\}}\big]=\mu_x(\varphi)\int_0^{T_1}G'\big(t_1\mu_x(b)\big)\,\deriv t_1
\]
Then, we deduce that for all $\varepsilon>0$, there exists $n_\varepsilon\geq0$ such that for all $n\geq n_\varepsilon$
\begin{align*}
\Big\lvert\espn_{x}\big[f(\mathbf{X}_{\tau_1^n}^n,\ldots,\mathbf{X}_{\tau_{N+1}^n}^n)\,\mathbf{1}_{\{\tau_1^n\leq T_1\}}\ldots&\mathbf{1}_{\{\tau_{N+1}^n-\tau_{N}^n\leq T_{N+1}\}}\big]-\mu_x(\varphi)\int_0^{T_1}G'(t_1\mu_x(b))\,\deriv t_1\Big\rvert\\
&\leq\bigg\lvert\mathbf E_x^\ori\Big[\int_0^{T_1}\deriv t_1\,\big(G'(I_{t_1}^n(b))-G_\varepsilon'(I_{t_1}^n(b))\big)\,\varphi^n(X_{nt_1}^\ori)\Big]\bigg\rvert\\
&+\Big\lvert\mathbf E_x^\ori\Big[\int_0^{T_1}\deriv t_1\,\big(G_\varepsilon'(I_{t_1}^n(b))-G_\varepsilon'(t_1\mu_x(b))\,\varphi^n(X_{nt_1}^\ori)\big]\Big\rvert\\
&+\Big\lvert\mu_x(\varphi)\int_0^{T_1}\big(G_\varepsilon'(t_1\mu_x(b))-G'(t_1\mu_x(b))\big)\,\deriv t_1\Big\rvert\\
&\leq\varepsilon,
\end{align*}
which conclude the proof of the proposition.
\end{proof}

\begin{proof}[Proof of Corollary~\ref{cor:cv-I-jump_times}] 
Without risk of confusion, we denote, by $\phi$ the function defined on $(E^N,\mathcal E^N)$ taking values in $(I^N,\mathcal I^N)$, for all $N\geq1$, by
\[
    \phi(x_1,\ldots,x_N)=(i_1,\ldots,i_N)\quad\textup{if and only if}\quad (x_1,\ldots,x_N)\in E_{i_1}\times\ldots\times E_{i_N},
\]
with $\mathcal{E}^N=\otimes_{j=1}^N\mathcal E_j$ and $\mathcal{I}^N=\otimes_{j=1}^N\mathcal I_j$ where $\mathcal{E}_j=\mathcal{E}$ and $\mathcal{I}_j=\mathcal{I}$  for all $j=1,\ldots,N$.

Let $n,N\geq1$  positive integers, $T_N>\ldots>T_1>0$, strictly positive real numbers, a bounded measurable function $f:I^N\longrightarrow\mathbb R$, $i\in I$ and $x_i\in E_i$. We observe that
\begin{multline*}
\mathbf E_{x_i}^n\big[f\big(\mathbf{I}_{\tau_1^n}^n,\ldots,\mathbf{I}_{\tau_N^n}^n\big)\,\mathbf{1}_{\{\tau_1\leq T_1\}}\ldots\mathbf{1}_{\{\tau_{N}^n-\tau_{N-1}^n\leq T_N\}}\big]\\
=\mathbf E_{x_i}^n\big[f\circ\phi\big(\mathbf{X}_{\tau_1^n}^n,\ldots,\mathbf{X}_{\tau_N^n}^n\big)\,\mathbf{1}_{\{\tau_1^n\leq T_1\}}\ldots\mathbf{1}_{\{\tau_{N}^n-\tau_{N-1}^n\leq T_N\}}\big].
\end{multline*}
By definition of $\phi:E^N\longrightarrow I^N$ and that the function $f:I^N\to\mathbb R$ is bounded, the function $f\circ\phi:E^N\longrightarrow\mathbb R$ is bounded, so using the Theorem~\ref{thm:cv-X-jump_times} we get
\begin{multline*}
\lim_{n\to\infty}\mathbf E_{x_i}^n\big[f\big(\mathbf{I}_{\tau_1^n}^n,\ldots,\mathbf{I}_{\tau_N^n}^n\big)\,\mathbf{1}_{\{\tau_1\leq T_1\}}\ldots\mathbf{1}_{\{\tau_{N}^n-\tau_{N-1}^n\leq T_N\}}\big]\\
=\mathbf E_{x_i}\big[f\circ\phi(\mathbf Y_1,\ldots,\mathbf Y_N)\,\mathbf 1_{\{\tau_1\leq T_1\}}\ldots\mathbf 1_{\{\tau_N-\tau_{N-1}\leq T_N\}}\big].
\end{multline*}
Reminding the notation $\mu_{x_i}=\mu_i=\mu_{\phi(x_i)}$ if and only if $x_i\in E_i$ for all $i\in I$, we get the previous term is none other than quantity 
\[
    \mathbf{E}_{i}\big[f\big(\mathbf I_{\tau_1},\ldots,\mathbf I_{\tau_N}\big)\,\mathbf{1}_{\{\tau_1\leq t_1\}}\ldots\mathbf{1}_{\{\tau_{N}-\tau_{N-1}\leq t_N\}}\big]
\]
and conclude the proof of the theorem.
\end{proof}

\appendix

\section{Technical lemmas}
\label{sec:App-technical}
\begin{lemma}
\label{lem:mesE}
A function $f$ on $(E,\mathcal{E})$ is $\mathcal{E}$-measurable if and only if for all $i\in I$ its restriction $f_i$ to $E_i$ is $\mathcal{E}_i$-measurable. 
\end{lemma}

\begin{proof}[Proof of Lemma~\ref{lem:mesE}]
Indeed, letting $f:E\to\mathbb{R}$ a measurable function on each $E_i$ with $i\in I$, we show that $f$ is $\mathcal{E}$-measurable. Let $B\in\mathcal{B}([0,\infty))
$, we have
\[
f^{-1}(B)=\{x\in E:f(x)\in B\}=\{x\in\cup_{i\in I}E_i:f(x)\in B\}=\cup_{i\in I}\{x\in E_i:f(x)\in B\},
\]
then for all $i\in I$, $
f^{-1}(B)\cap E_i=\{x\in E_i:f(x)\in B\}$, so $f^{-1}(B)\cap E_i\in\mathcal{E}_i$ because $f$ is $\mathcal{E}_i$-measurable for all $i\in I$. Then, $f$ is $\mathcal{E}$-measurable. Now, let $f:(E,\mathcal{E})\to\mathbb{R}$ a measurable function on $E$. For all $i\in I$, let $f_i:(E_i,\mathcal{E}_i)\to\mathbb{R}$ the restriction of $f$ to $E_i$. We show that $f_i$ is also $\mathcal{E}$-measurable. Let $B\in\mathcal{B}_{[0,\infty)}
$
\[
f_i^{-1}(B)=\{x\in E_i:f(x)\in B\}=\{x\in E:f(x)\in B\}\cap E_i.
\]
As, $f$ is $\mathcal{E}$-measurable, $\{x\in E:f(x)\in B\}\in\mathcal{E}$, then by definition of $\mathcal{E}$, $\{x\in E:f(x)\in B\}\in\mathcal{E}\cap E_i\in\mathcal{E}_i$ for all $i\in I$. Thus $f_i$ is $\mathcal{E}_i$-measurable for all $i\in I$.
\end{proof}

\subsection{Tensorisation of the total variation convergence}

\begin{lemma}
Let $(\mu_n)_{n\geq1}$ and $(\nu_n)_{n\geq1}$ two sequences of probability measures. Then, for $n\geq2$ we have
\[
    \|\otimes_{k=1}^n\mu_k-\otimes_{k=1}^n\nu_k\|_{\mathrm{TV}}\leq\sum_{k=1}^n\|\mu_k-\nu_k\|_{\mathrm{TV}}.
\]

\end{lemma}
\begin{proof}
For the proof of this lemma we use the upper bound in (1.4) from \cite{Kontorovic} with $P_k=\mu_k$ et $Q_k=\nu_k$ for $k=1,\ldots,n$.
\end{proof}

\begin{lemma}
\label{lem:tensorisation-cv-VT}
Let $(X^n)_{n\geq1}$ a sequence of Markov processes on $E_n$ with semi-group $P^n$. We assume that for all $n\geq1$ and $x\in E_n$ there exists a probability measure $\mu_n$ such that
\[
\|p_t^n(x,\,\cdot\,)-\mu_n\|_{\mathrm{TV}}\underset{t\to\infty}{\longrightarrow}0.
\]
Let $Z=(X^1,\ldots,Y^n)_{t\geq0}$ a process on $E_1\times\ldots\times E_n$, with semi-group $(Q_t)_{t\geq0}$ such that 
\[Q_t(f_1\ldots f_n)(x_1,\ldots, x_n)=P_t^1f_1(x_1)\cdots P_t^nf_n(x_n)\]
for all $(x_1,\ldots,x_n)\in E_1\times \ldots\times E_n$ and for all functions  $f_i\in B_b(E_i)$ $(i=1,\ldots,n)$. Then for all $z=(x_1,\ldots,x_n)\in E_1\times\ldots\times E_n$
\begin{align*}
\|q_t(z,\,\cdot\,)-\mu_1\otimes\ldots\otimes\mu_n\|_{\mathrm{TV}}\underset{t\to\infty}{\longrightarrow}0.
\end{align*}
\end{lemma}

\begin{proof}
Let $\nu=\nu_1\otimes\ldots\otimes\nu_n$ a probability measure on $E_1\ldots\times\ldots E_n$ and $f:E_1\times\ldots\times E_n\longrightarrow[0,\infty]$ such that $f=f_1\cdots f_n$ with $f_i:E_i\longrightarrow[0,\infty]$ ($i=1,\ldots, n)$. Then, for all $t\geq0$
\begin{align*}
\nu\,Q_tf&=\int_{E_1\times\ldots\times E_n}Q_tf(x)\,\nu(\mathrm d x)\\
&=\int_{E_1\times\ldots\times E_n}Q_tf(x)\,\nu_1\otimes\ldots\otimes\nu_n(\mathrm dx)\\
&=\int_{E_1\times\ldots\times E_n}P_t^1f_1(x_1)\ldots P_t^nf_n(x_n)\,\nu_1(\mathrm d x_1)\ldots\nu_n(\mathrm d x_n)\\
&=\nu_1\,P_t^1\otimes\ldots\otimes\nu_n\,P_t^n f
\end{align*}
so
\[
\nu\,Q_tf-\mu_1\otimes\ldots\otimes\mu_n(f)=\nu_1\,P_t^1\otimes\ldots\otimes\nu_n\,P_t^n(f)-\mu_1\otimes\ldots\otimes\mu_n(f).
\]
Then, 
\[
\|\nu\,Q_t-\mu_1\otimes\ldots\otimes\mu_n\|_{\mathrm{TV}}\leq\|\nu_1\,P_t^1-\mu_1\|_{\mathrm{TV}}+\ldots+\|\nu_n\,P_t^n-\mu_n\|_{\mathrm{TV}}
\]
which goes to zero, using the previous lemma.
\end{proof}

\section{Construction of the semi-Markov process $\mathbf X^1$}

\label{sec:appendixPO}

For the construction of the process $\mathbf X^1$, we take inspiration from the piecing-out construction in~\cite{MBP2}.

Let $(V,\mathcal F^\ori)$ be a measurable space endowed with a filtration $(\mathcal F^\ori_t)_{t\geq 0}$ and a family of probability measures $(\mathbf{P}_x^\ori)_{x\in E}$. We assume that we are given a progressively measurable  process 
\[
X^\ori=(V, \mathcal F^\ori,(\mathcal F^\ori_t)_{t\geq 0},(\mathbf{P}_x^\ori)_{x\in E},(X_t^\ori)_{t\geq0})
\]Consider the measurable probability  space $(\tilde\Omega,\tilde{\mathcal G},\tilde{\mathbf P}_x)$, where
$\tilde\Omega=V\times [0,1]\times E$, $\tilde{\mathcal G}=\mathcal F^\ori\otimes \mathcal B([0,1])\otimes \mathcal E$ and
\[
\tilde{\mathbf P}_x(A\times B\times C)=\int_{A\times B}\mathbf{P}^\ori_x(d v)\,\lambda(d\,u)\,\pi(X^\ori_{\zeta(v,u)}(v),C),
\]
with
\[
\zeta(v,u)=\inf\{t\geq 0:  G\Big(\int_0^t \,b(X_s^{\ori}(v))\,\deriv s\Big)> u\}.
\]

\begin{remark}
The random variable $\zeta$ has the same law as $\inf\{t\geq 0, \ \int_0^t b(X_s^{\ori})\,\mathrm ds \geq U'\}$, where $U'$ is a random variable with repartition function $G$. In the particular situation where $G(x)=1-e^{-x}$, we recover the classical case of a process with exponential jump rate $b$.
\end{remark}

\begin{remark}
The remaining of the paper can be easily adapted to the situation where $G$ depends on $\phi(X^\ori_0)$, at the expense of heavier notations. This would be relevant, e.g. for PDMPs. However, for the sake of readability, we keep $G$ independent of $\phi(X^\ori_0)$ in the rest of the paper.
\end{remark}

We set, for all $t\geq 0$, $\tilde{\mathcal G}_t=\mathcal F_t^\ori\otimes \mathcal B([0,1])\otimes \mathcal E$. We are now in position to define the progressively mesurable process
\[
X=(\tilde\Omega, \tilde{\mathcal G},(\tilde{\mathcal G}_t)_{t\geq 0},(\tilde{\mathbf{P}}_x)_{x\in E},(X_t)_{t\geq0})
\]
with values in $(E,\mathcal E)$ an such that
\[
X_t(v,u,y)=
\begin{cases}
X_t^{\ori}(v)\quad&\textup{si}\quad t<\zeta(v,u),\\
y\quad&\textup{sinon}.
\end{cases}
\]

We also consider the set $\Omega:=\prod_{j=1}^\infty \tilde\Omega_j$ (where $\tilde\Omega_j=\tilde\Omega,\,j=1,2,\dots$) endowed with the product $\sigma$-field $\mathcal{F}=\otimes_{j=1}^\infty\tilde{\mathcal{G}}_j$,
$(\textup{where }\:\tilde{\mathcal{G}}_j=\tilde{\mathcal{G}},\,j=1,2,\ldots)$. According to Ionescu-Tulcea Theorem (see Theorem~8.24, \cite{kallenberg2002foundations}), there exists a unique system $(\mathbf P_x)_{x\in E}$ of probability measures $(\Omega,\mathcal{F})$ such that, for all non-negative measurable function $F$ on $(\prod_{j=1}^n \tilde\Omega_j,\otimes_{j=1}^n\tilde{\mathcal{G}}_j)$, $(n=1,2,\ldots)$,
\begin{equation}
\mathbf E_x[F(\omega_1,\ldots,\omega_n)]=\int_{\Omega^n}\tilde{\mathbf P}_x(\deriv\omega_1)\times\ldots\times \tilde{\mathbf P}_{y_{n-1}}(\deriv\omega_n)\,F(\omega_1,\omega_2,\ldots,\omega_n),
\label{eq:mesureProd}
\end{equation}
with $\omega_j=(v_j,u_j,y_j)$.

We define the process $\mathbf X^1=(\mathbf{X}_t^1)_{t\geq0}$ on $(\Omega,\mathcal{F})$ with values in $E$ by piecing out as follows. For any $\omega=(\omega_1,\omega_2,\ldots)\in\Omega$, où $\omega_j=(v_j,u_j,y_j)$, 
we define  $\mathbf{X}_t^1(\omega)$ by
\[
\mathbf{X}_t^1(\omega)=
\begin{cases}
X_t(\omega_1),&\textup{if}\;0\leq t\leq\zeta(v_1,u_1),\\
X_{t-\zeta(v_1,u_1)}(\omega_2),&\textup{if}\;\zeta(v_1,u_1)<t\leq\zeta(v_1,u_1)+\zeta(v_2,u_2),\\
\vdots&\vdots\\
X_{t-(\zeta(v_1,u_1)+\zeta(v_2,u_2)+\ldots+\zeta(v_n,u_n))}(\omega_{n+1}),&\textup{if}\;\sum_{j=1}^{n}\zeta(v_j,u_j)<t\leq\sum_{j=1}^{n+1}\zeta(v_j,u_j),\\
\vdots&\vdots\\
\Delta&\textup{if}\; t\geq\overset{\infty}{\underset{j=1}{\sum}}\zeta(v_j,u_j).
\end{cases}
\]
where $\Delta \notin E$ is a cemetery point.
The life time of $\zeta(\omega)$ of $\mathbf{X}_t(\omega)$ is then defined by
\[
\zeta(\omega)=\sum_{j=1}^{\infty}\zeta(v_j,u_j).
\]

\begin{remark}
Under the probability law $\tilde{\mathbf P}$, we have $X_0(\omega_2) = y_1 = X_{\zeta((v_1, u_1)}(\omega_1)$ so that $\mathbf{X}^1$ is right continuous at time $\zeta(v_1, u_1)$ if $X^\ori$ is itself right continuous.
\end{remark}

In addition, we introduce the sequence of random times $(\tau_k(\omega))_{k\geq0}$
\[
\tau_0(\omega)=0,\,\tau(\omega)\equiv\tau_1(\omega)=\zeta(v_1,u_1),\,\ldots\,,\tau_k(\omega)=\sum_{j=1}^{k}\zeta(v_j,u_j).
\]

\begin{proposition}
\label{prop:piecingOut}
The process $\mathbf{X}^1$ previously defined is a process on the probability space $\Omega$, taking values in $E\cup\{\Delta\}$ absorbed in state $\Delta$ such that,
such that for all $t\geq 0$, $B\in\mathcal{F}$ and $A\in\mathcal{E}$, 
\begin{align*}
\mathbf P_x(\{\omega:v_1\in B,\;\tau_1(\omega)\leq t,\textup{ and }\mathbf{X}_{\tau_1}^1(\omega)\in A\})=\int_B\mathbf{P}^\ori_x(\deriv v)\,\int_0^1 \mathrm du\,\mathbf 1_{\{\zeta(v,u)\leq t\}}\,\pi(X^\ori_{\zeta(v,u)}(v),A),
\end{align*}
where we wrote $\omega=(\omega_1,\omega_2,\ldots)$ et $\omega_j=(v_j, u_j,y_j)$. 
\end{proposition}
\begin{proof}
It is sufficent to apply equation \eqref{eq:mesureProd} to $F: \Omega_1 \longrightarrow\mathbb{R}$ defined by
\[
F(v_1, u_1, y_1) = \mathbf 1_B(v_1)\,\mathbf 1_{\{\zeta(v_1, u_1)\leq t\}} \,\mathbf 1_A(y_1)
\]
\end{proof}
For all $k\geq 1$, we note by $\varphi_k$ the projection of $\Omega$ on $\prod_{j=1}^k\tilde\Omega_j$ ($\tilde\Omega_j=\tilde\Omega)$. Then, we defined
\begin{align*}
\mathcal{F}_{\tau_k}=\varphi_k^{-1}(\otimes_{j=1}^k\tilde{\mathcal{G}})
\end{align*}
The process defined above verifies the Markov property at jump times $(\tau_k)_{k\geq0}$. For all $k\geq0$, we now have to defined the shift operator  $\theta_{\tau_k}:\Omega\to\Omega$ as follow: for $\omega=(\omega_1,\omega_2,\omega_3,\ldots),$
\[
\theta_{\tau_k}\omega=
(\omega_{k+1},\omega_{k+2},\omega_{k+3},\ldots)
\]
\begin{proposition}
\label{prop:propMarkov}
\begin{enumerate}
\item For all $B\in\mathcal{F}$ and $A\in\mathcal{F}_{\tau_k}$,
\begin{equation}
\tilde{\mathbf P}_x(A,\,\theta_{\tau_k}\omega\in B)=\mathbf E_x\big[\mathbf{1}_A\,\mathbf P_{\mathbf X_{\tau_k}^1}(B)\big].
\tag{1.5.1}
\label{eq:propMarkov1}
\end{equation}
\label{lem:propMarkov1}
\item Let $g(\omega,t)$ a bounded $\mathcal{F}\otimes\mathcal{B}([0,\infty])$-measurable function on $\Omega\times[0,\infty]$. If $\tau(\omega)\geq0$ is $\mathcal{F}_{\tau_k}$-measurable and $A\in\mathcal{F}_{\tau_k}$,
\begin{equation}
\mathbf E_x[\mathbf{1}_A\,g(\theta_{\tau_k}\omega,\tau)]=\mathbf E_x\big[\mathbf{1}_A\,\mathbf E_{\mathbf X_{\tau_k}^1}[g(\cdot,s)]|_{s=\tau}\big].
\tag{1.5.2}
\label{eq:propMarkov2}
\end{equation}
\label{lem:propMarkov2}
\item Let $g(\omega,\omega')$ a bounded $\mathcal{F}_{\tau_k}\otimes\mathcal{F}$-measurable function on $\Omega\otimes\Omega$. Then, for all $A\in\mathcal{F}_{\tau_k}$,
\begin{equation}
\mathbf E_x[\mathbf{1}_A\,g(\tilde{\omega},\theta_{\tau_k}\omega)]=\mathbf E_x\big[\mathbf{1}_A\,\mathbf E_{\mathbf X_{\tau_k}^1}[g(u,\cdot)]|_{u=\omega}\big].
\tag{1.5.3}
\label{eq:propMarkov3}
\end{equation}
\label{lem:propMarkov3}
\end{enumerate}
\end{proposition}
\begin{proof}
Let $(A_j)_{1\leq j\leq n}\subset\tilde{\mathcal{G}}$ a sequence of  $\tilde{\mathcal{G}}$. We first show that \ref{lem:propMarkov1} for $A=\{\omega:\omega_1\in A_1,\ldots,\omega_k\in A_k\}$ and $B=\{\omega:\omega_1\in A_{k+1},\ldots,\omega_{n-k}\in A_n\}$. From the definition of $\mathbf P_x$

\begin{align*}
\mathbf P_x\big(\{\omega:&\omega_1\in A_1,\ldots,\omega_n\in A_n\}\big)=\int_{A_1}\ldots\int_{A_n}Q(x,\deriv\omega_1)\ldots Q(y_{n-1},\deriv\omega_n)\\
&=\int_{A_1}\ldots\int_{A_k}Q(x,\deriv\omega_1)\ldots Q(y_{k-1},\deriv\omega_k)\int_{A_{k+1}}\ldots\int_{A_n}Q(y_k,\deriv\omega_{k+1})\ldots Q(y_{n-1},\deriv\omega_n)\\
&=\int_{A_1}\ldots\int_{A_k}Q(x,\deriv\omega_1)\ldots Q(y_{k-1},\deriv\omega_k)\,\mathbf P_{y_k}\big(\{\omega':\omega_1'\in A_{k+1},\ldots,\omega_{n-k}'\in A_n\}\big)\\
&=\mathbf E_x\Big[\mathbf{1}_{A_1\times\ldots\times A_k}(\omega)\,\mathbf P_{\mathbf{X}_{\tau_k}^1}\big(\{\omega:\omega_1\in A_{k+1},\,\omega_2\in A_{k+2},\ldots,\omega_{n-k}\in A_n\}\big)\Big]
\end{align*}
with $\omega_j=(w_j,y_j)$ for all  $j=1,2,\ldots,n$. We set $\mathcal{M}$ the set's collection of  $\mathcal{F}_{\tau_k}\otimes\mathcal{F}$ defined by
\[
\mathcal{M}=\big\{(A,B)\in\mathcal{F}_{\tau_k}\otimes\mathcal{F}:\mathbf P_x(A,\theta_{\tau_k}\omega\in B)=\mathbf E_x[\mathbf{1}_A\,\mathbf P_{\mathbf X_{\tau_k}^1}(B)]\big\}.
\]
$\mathcal{M}$ is a $\lambda$-system. Indeed, let $(A_1,B_1)\in\mathcal{M}$ and $(A_2,B_2)\in\mathcal{M}$ such that $A_1\subset A_2$ et $B_1\subset B_2$ so
\begin{align*}
\mathbf P_x(A_2\setminus A_1,\tilde{\theta}_{\tau_k}\omega\in B_2\setminus B_1)&=\mathbf P_x(A_2,\theta_{\tau_k}\omega\in B_2\setminus B_1)-\mathbf P_x(A_1,\theta_{\tau_k}\omega\in B_2\setminus B_1)\\
&=\mathbf P_x(A_2,\theta_{\tau_k}\omega\in B_2)-\mathbf P_x(A_2,\theta_{\tau_k}\omega\in B_1)\\
&-\mathbf P_x(A_1,\theta_{\tau_k}\omega\in B_2)+\mathbf P_x(A_1,\theta_{\tau_k}\omega\in B_1)\\
&=\mathbf E_x[\mathbf{1}_{A_2}\,\mathbf P_{\mathbf X_{\tau_k}^1}(B_2)]-\mathbf E_x[\mathbf{1}_{A_2}\,\mathbf P_{\mathbf X_{\tau_k}^1}(B_1)]\\
&-\mathbf E_x[\mathbf{1}_{A_1}\,\mathbf P_{\mathbf X_{\tau_k}^1}(B_2)]+\mathbf E_x[\mathbf{1}_{A_1}\,\mathbf P_{\mathbf X_{\tau_k}^1}(B_1)]\\
&=\mathbf E_x[\mathbf{1}_{A_2}\,\mathbf P_{\mathbf X_{\tau_k}^1}(B_2\setminus B_1)]-\mathbf E_x[\mathbf{1}_{A_1}\,\mathbf P_{\mathbf X_{\tau_k}^1}(B_2\setminus B_1)]\\
&=\mathbf E_x[\mathbf{1}_{A_2\setminus A_1}\,\mathbf P_{\mathbf X_{\tau_k}^1}(B_2\setminus B_1)].
\end{align*}
Then $(A_2\setminus A_1,B_2\setminus B_1)\in\mathcal{M}$. Let $n\geq 0$ and $(A_n,B_n)\in\mathcal{M}$ such that $(A_n,B_n)\subset (A_{n+1},B_{n+1})$. We set $A=\underset{n\geq0}{\bigcup}A_n$ and $B=\underset{n\geq0}{\bigcup}B_n$. Let us show that $(A,B)\in\mathcal{M}$
\begin{align*}
\mathbf P_x(A,\theta_{\tau_k}\omega\in B)&=\lim_{n\to\infty}\mathbf P_x(A_n,\theta_{\tau_k}\omega\in B)\\
&=\lim_{n\to\infty}\lim_{m\to\infty}\mathbf P_x(A_n,\tilde{\theta}_{\tau_k}\omega\in B_m)\\
&=\lim_{n\to\infty}\lim_{m\to\infty}\mathbf E_x[\mathbf{1}_{A_n}\,\mathbf P_{\mathbf X_{\tau_k}^1}(B_m)]\\
&=\lim_{n\to\infty}\mathbf E_x[\mathbf{1}_{A_n}\,\mathbf P_{\mathbf X_{\tau_k}^1}(B)]\\
&=\mathbf E_x[\mathbf{1}_{A_n}\,\mathbf P_{\mathbf X_{\tau_k}^1}(B_m)]\\
&=\lim_{n\to\infty}\mathbf E_x[\mathbf{1}_A\,\mathbf P_{\mathbf X_{\tau_k}^1}(B)].
\end{align*}
We have shown that $\mathcal{M}$ is a $\lambda$-system. However by definition, $\mathcal{F}_{\tau_k}$ is the $\sigma$-algebra generated by the $\pi$-system
\[
\mathcal{C}_k=\{A_1\times A_2\times\ldots\times A_k|A_1\in\tilde{\mathcal{G}},A_2\in\tilde{\mathcal{G}},\ldots,A_k\in\tilde{\mathcal{G}}\}
\]
and $\mathcal{F}$ by the $\pi$-système
\[
\mathcal{C}=\{A_1\times A_2\times\ldots\times A_n|A_1\in\tilde{\mathcal{G}},A_2\in\tilde{\mathcal{G}},\ldots,A_n\in\tilde{\mathcal{G}}\,\textup{pour tout}\,n\geq1\}.
\]
Yet $\mathcal{C}_k\otimes\mathcal{C}\subset\mathcal{M}$ so by the $\pi$-$\lambda$-theorem $\mathcal{M}=\mathcal{F}_{\tau_k}\otimes\mathcal{F}$ as $\pi(\mathcal{C}_k\otimes\mathcal{C})=\mathcal{F}_{\tau_k}\otimes\mathcal{F}$. So we have shown (\ref{eq:propMarkov1}). By a $\pi$-$\lambda$-theorem argument (\ref{eq:propMarkov2}) follows from (\ref{eq:propMarkov1}). Same for \ref{lem:propMarkov3}, we use (\ref{eq:propMarkov1}) to verify (\ref{eq:propMarkov3}) assuming that $g(\omega,\omega')=g_1(\omega)g_2(\omega')$ où $g_1$ is a bounded $\mathcal{F}_{\tau_k}$-measurable function and $g_2$ is a bounded $\mathcal{F}$-measurable function. Also by $\pi$-$\lambda$-theorem argument we get (\ref{eq:propMarkov3}) for all bounded $\mathcal{F}_{\tau_k}\otimes\mathcal{F}$-measurable function $g$.
\end{proof}

\end{document}